\documentclass[12pt]{amsart}
\usepackage{amsmath,amsthm,amsfonts,amssymb,amscd,euscript,enumerate}
\usepackage{color,listings,graphics,tikz,epsfig} 
\input{xy}
\xyoption{all}

\newlength{\defbaselineskip}
\theoremstyle{plain}
\newtheorem{theorem}{Theorem}[section]
\newtheorem{proposition}[theorem]{Proposition}

\newtheorem{lemma}[theorem]{Lemma}

\theoremstyle{definition}
\newtheorem{definition}[theorem]{Definition}

\newtheorem{remark}[theorem]{Remark}
\newtheorem{example}[theorem]{Example}

\newtheorem{conjecture}[theorem]{Conjecture}

\newcommand{\calB}{\mathcal{B}}
\newcommand{\C}{\mathcal{C}}

\newcommand{\LL}{\mathcal{L}}
\newcommand{\calS}{\mathcal{S}}
\newcommand{\calT}{\mathcal{T}}
\newcommand{\Z}{\mathbb{Z}}

\newcommand{\maxk}{11}

\DeclareMathOperator{\area}{area}
\DeclareMathOperator{\arm}{arm}

\DeclareMathOperator{\Cat}{Cat}
\DeclareMathOperator{\Def}{Def}
\DeclareMathOperator{\defc}{defc}
\DeclareMathOperator{\dg}{dg}
\DeclareMathOperator{\dinv}{dinv}

\DeclareMathOperator{\leg}{leg}

\DeclareMathOperator{\Par}{Par}

\DeclareMathOperator{\lpart}{\textsc{left}}
\DeclareMathOperator{\mpart}{\textsc{mid}}
\DeclareMathOperator{\rpart}{\textsc{right}}
\DeclareMathOperator{\tail}{\textsc{tail}}
\DeclareMathOperator{\TI}{TI}

\newcommand{\DV}{\mathcal{DV}} 
 
\newcommand{\dvmap}{\textsc{dv}}
\newcommand{\DP}{\mathcal{DP}} 
\newcommand{\dpmap}{\textsc{dp}}

\newcommand{\mind}{\text{min}_{\Delta}}

\numberwithin{equation}{section}

\definecolor{darkgreen}{rgb}{0.0, 0.7, 0.0}

\definecolor{cyan}{cmyk}{1,0,0,0}

\newlength{\cellsize}
\newcommand\tableau[1]{
\vcenter{
\let\\=\cr
\baselineskip=-16000pt
\lineskiplimit=16000pt
\lineskip=0pt
\halign{&\tableaucell{##}\cr#1\crcr}}}

\newcommand{\tableaucell}[1]{{%
\def \arg{#1}\def \void{}%
\ifx \void \arg
\vbox to \cellsize{\vfil \hrule width \cellsize height 0pt}%
\else
\unitlength=\cellsize
\begin{picture}(1,1)
\put(0,0){\makebox(1,1){$#1$}}
\put(0,0){\line(1,0){1}}
\put(0,1){\line(1,0){1}}
\put(0,0){\line(0,1){1}}
\put(1,0){\line(0,1){1}}
\end{picture}%
\fi}}
\begin{document}

\title{Chain Decompositions of $q,t$-Catalan Numbers via Local Chains}
\subjclass[2010]{05A19, 05A17, 05E05}
\date{\today} 

\author{Seongjune Han}
\address{Department of Mathematics \\
University of Alabama \\
Tuscaloosa, AL 35401}
\email{shan25@crimson.ua.edu}

\author{Kyungyong Lee}
\address{Department of Mathematics \\
University of Alabama \\
Tuscaloosa, AL 35401 USA; \\
and Korea Institute for Advanced Study \\
Seoul 02455, Republic of Korea}
\email{kyungyong.lee@ua.edu; klee1@kias.re.kr}
\thanks{
The second author was supported by NSF grant DMS 1800207, the Korea Institute for
Advanced Study (KIAS), and the University of Alabama.}

\author{Li Li}
\address{Department of Mathematics and Statistics \\
 Oakland University \\ 
 Rochester, MI 48309}
\email{li2345@oakland.edu} 

\author{Nicholas A. Loehr}
\address{Department of Mathematics \\
 Virginia Tech \\
 Blacksburg, VA 24061-0123} %
\email{nloehr@vt.edu} 
\thanks{This work was supported by a grant
from the Simons Foundation/SFARI (Grant \#633564 to N.A.L.).}

\keywords{$q,t$-Catalan numbers, Dyck paths, dinv statistic, joint symmetry,
integer partitions, chain decompositions}
\begin{abstract}
The $q,t$-Catalan number $\Cat_n(q,t)$ enumerates integer partitions contained
in an $n\times n$ triangle by their dinv and external area statistics.
The paper~\cite{LLL18} proposed a new approach to understanding the
symmetry property $\Cat_n(q,t)=\Cat_n(t,q)$ based on decomposing the set
of all integer partitions into infinite chains. Each such global chain
$\C_{\mu}$ has an opposite chain $\C_{\mu^*}$; these combine to give a new 
small slice of $\Cat_n(q,t)$ that is symmetric in $q$ and $t$.
Here we advance the agenda of~\cite{LLL18} by developing a new general
method for building the global chains $\C_{\mu}$ from smaller elements
called \emph{local chains}. We define a \emph{local opposite property}
for local chains that implies the needed opposite property of the global 
chains. This local property is much easier to verify in specific cases
compared to the corresponding global property. We apply this machinery
to construct all global chains for partitions with deficit at most $\maxk$.
This proves that for all $n$, the terms in $\Cat_n(q,t)$ of degree
at least $\binom{n}{2}-\maxk$ are symmetric in $q$ and $t$.  
\end{abstract}
\maketitle

\section{Introduction}
\label{sec:intro}

The \emph{$q,t$-Catalan numbers} $\Cat_n(q,t)$ are polynomials in $q$ and $t$
that reduce to the ordinary Catalan numbers when $q=t=1$. These polynomials
play a prominent role in modern algebraic combinatorics, with connections
to representation theory, algebraic geometry, symmetric functions, knot
theory, and other areas. Garsia and Haiman~\cite{GH-qtcat} originally defined 
these polynomials as sums of complicated rational functions indexed
by integer partitions. Haglund~\cite{hag-qtcatconj} 
and Haiman independently discovered
elegant combinatorial interpretations of the $q,t$-Catalan numbers as weighted
sums of Dyck paths. Garsia and Haglund~\cite{GH-qtcatpf} proved that Haglund's
combinatorial formula was equivalent to the original definition. 
More background on $q,t$-Catalan numbers may be found in Haglund's 
book~\cite{hag-book} and in~\cite[Sec. 1]{LLL13}.

One version of the combinatorial formula for $\Cat_n(q,t)$ is a sum 
over Dyck paths weighted by statistics called \emph{area} and \emph{dinv}. 
We can regard a Dyck path as the southeast border of a partition diagram 
contained in the triangle $\Delta_n$ with vertices $(0,0)$, $(0,n)$, and 
$(n,n)$.  This lets us rewrite the formula for $\Cat_n(q,t)$ as a weighted sum 
over all integer partitions that fit in this triangle:
\begin{equation}\label{eq:def-qtcat}
 \Cat_n(q,t)=\sum_{\gamma\subseteq \Delta_n}
    q^{|\Delta_n|-|\gamma|}t^{\dinv(\gamma)}. 
\end{equation}
(See Section~\ref{subsec:ptn-stats} for the definition 
of $\dinv(\gamma)$ and other notation used in this formula.) 

It is known~\cite{GH-PNAS,GH-qtcatpf} 
(see also~\cite{carlsson-mellit,Haiman,Mellit})
that $\Cat_n(q,t)=\Cat_n(t,q)$ for every $n$, but it is a 
notoriously difficult open problem to give a combinatorial proof of this fact 
based on~\eqref{eq:def-qtcat} or related formulas. In~\cite{LLL18}, the last
three authors proposed an approach to this problem based on the following
ideas. Instead of focusing only on integer partitions contained in a
particular triangle $\Delta_n$, we consider the infinite set $\Par$ of all
integer partitions. We seek to decompose this set into a disjoint union
of \emph{chains} denoted $\C_{\mu}$, where each chain is indexed by
an integer partition $\mu$ called a \emph{deficit partition}. 
Each chain is an infinite sequence of
partitions such that dinv increases by $1$ as we move along the chain.
Moreover, for each $\gamma$ in the chain $\C_{\mu}$,
the \emph{deficit} statistic $\defc(\gamma)=|\gamma|-\dinv(\gamma)$ 
has the constant value $|\mu|$. Among other technical conditions,
the chains $\C_{\mu}$ must satisfy the following crucial \emph{opposite
property}. For each $n\geq 0$ and collection $\calS$ of partitions, define 
\begin{equation}\label{eq:Cat-n-mu}
 \Cat_{n,\calS}(q,t)=\sum_{\gamma\in\calS:\,\gamma\subseteq\Delta_n}
   q^{|\Delta_n|-|\gamma|}t^{\dinv(\gamma)}. 
\end{equation}
The opposite property asserts that for each $k$, there is an involution
$\mu\mapsto\mu^*$ on the set of partitions of $k$ such that for every $n\geq 0$,
\[ \Cat_{n,\C_{\mu}}(q,t)=\Cat_{n,\C_{\mu^*}}(t,q). \]
If such chains $\C_{\mu}$ can be constructed
for all partitions $\mu$ of a fixed $k$, then we can deduce
the joint symmetry of the terms in $\Cat_n(q,t)$ of degree $\binom{n}{2}-k$.
At a finer level, every pair $\C_{\mu}$ and $\C_{\mu^*}$ that we build
reveals a new ``small slice'' of the Catalan objects that is symmetric
in $q$ and $t$.  A remarkable feature of this setup is that the 
infinite chains $\C_{\mu}$ and $\C_{\mu^*}$ (which do not depend
on $n$) induce joint symmetry for all $n$ simultaneously.

Here is a brief summary of the main results in~\cite{LLL18} most relevant 
to our current work. Conjecture~6.9 of~\cite{LLL18} gives a complete
technical statement of the decomposition of $\Par$ into the chains
$\C_{\mu}$ outlined above. A version of this conjecture appears as
Conjecture~\ref{conj:global} below.
Section~2 of~\cite{LLL18} explicitly constructs the chains $\C_{\mu}$
for one-row partitions $\mu=(k)$. In this case, $\mu^*=\mu$, and
the self-opposite property $\Cat_{n,\C_{(k)}}(q,t)=\Cat_{n,\C_{(k)}}(t,q)$ is
proved in Section~3 of~\cite{LLL18}. Section~4 of~\cite{LLL18} constructs
the chains $\C_{\mu}$ for two-row partitions of the form
$\mu=(ab-b-1,b-1)$ and $\mu^*=(ab-a-1,a-1)$. The opposite property
for these chains is proved in Section~5 of~\cite{LLL18}. Finally,
with the aid of results in Section 6 and extensive computer calculations,
the online version of the appendix to~\cite{LLL18} presents
maps $\mu\mapsto\mu^*$ and chains $\C_{\mu}$ for all integer
partitions $\mu$ of size at most $9$. However, it should be emphasized
that the chains in this appendix were found through exhaustive computer
searches, not by any systematic construction. These searches become
impractical for $k\geq 10$.

The main contribution of this paper is a new general method for building the 
\emph{global} chains $\C_{\mu}$ by piecing together smaller \emph{local} chains.
The precise definition of a local chain is rather technical (see 
Section~\ref{subsec:ord-local-chain}), but here is the rough idea. 
For each partition $\gamma$, we must keep
track of the \emph{least} integer $n$ such that $\gamma\subseteq\Delta_n$;
this integer is denoted $\mind(\gamma)$. A local chain is a sequence
of partitions $\gamma(a),\gamma(a+1),\ldots,\gamma(b)$ such that 
$\dinv(\gamma(i))=i$ for $a\leq i\leq b$, $\defc(\gamma(i))$ is constant
for $a\leq i\leq b$, and the sequence $(\mind(\gamma(i)):a\leq i\leq b)$
has a certain staircase structure (described later).
We show how suitable local chains may be pasted together to form global 
chains. We introduce the idea of locally opposite local chains
and use this concept to prove the needed opposite property of the
global chains $\C_{\mu}$ and $\C_{\mu^*}$. The new local framework leads to 
much shorter and conceptually simpler proofs of the opposite property, 
compared to the very intricate computations that were given in Sections 3
and 5 of~\cite{LLL18}.  We give a new conjecture on writing
the set $\Par$ as a union of (partially overlapping) local chains,
and we prove that this conjecture implies the earlier conjecture
on the decomposition of $\Par$ into global chains. Finally, we construct
global chains $\C_{\mu}$ satisfying the new local conjecture for
all partitions $\mu$ of size at most \maxk. 
In contrast to~\cite{LLL18}, these global chains were found not through
exhaustive computer searches but rather by applying systematic operations
for building local chains. The full technical details of
these operations (in their general form) will be the subject of a future paper.

The rest of this article is organized as follows.
Section~\ref{sec:background} reviews the needed background material 
and definitions, which are included so that 
this paper can be read independently of~\cite{LLL18},
On the other hand, to avoid undue repetition of technical details,
we do refer to~\cite{LLL18} for the proofs of some specific results.
Section~\ref{sec:local-chains} develops the theory of local chains,
states the new structural conjecture for local chains, and
proves that this conjecture implies the previous conjecture for global chains.
Section~\ref{sec:specific-chains} presents global chains $\C_{\mu}$ 
for $|\mu|\leq \maxk$ 
and explains how to verify that these chains satisfy the local conjecture.  

\section{Background}
\label{sec:background}

This section reviews definitions and preliminary results on partitions,
Dyck vectors, and a map $\nu$ that is useful for constructing chains.

\subsection{Partition Statistics}
\label{subsec:ptn-stats}

An \emph{integer partition} is a weakly decreasing finite sequence of
positive integers. Given a partition $\gamma=(\gamma_1,\gamma_2,\ldots,
\gamma_s)$, let $\gamma_i=0$ for all $i>s$. Any of these zero parts may be
appended to the sequence $\gamma$ without changing the partition.
The \emph{length} of $\gamma$ is $\ell(\gamma)=s$, 
the number of strictly positive parts of $\gamma$. 
The \emph{diagram} of $\gamma$ is the set
\[ \dg(\gamma)=\{(i,j)\in\Z_{>0}\times\Z_{>0}:\ 
   1\leq i\leq\ell(\gamma),\ 1\leq j\leq \gamma_i\}. \]
We visualize the diagram as an array of left-justified unit squares
with $\gamma_i$ squares in the $i$th row from the top.
The \emph{conjugate partition} $\gamma'=(\gamma_1',\gamma_2',\ldots)$
is defined by letting $\gamma_j'$ be the number of cells
in the $j$th column of $\dg(\gamma)$, for $1\leq j\leq\gamma_1$.

The \emph{arm} of a cell $c=(i,j)$ in $\dg(\gamma)$ is
$\arm(c)=\lambda_i-j$, which is the number of cells strictly right
of $c$ in its row. The \emph{leg} of a cell $c=(i,j)$
in $\dg(\gamma)$ is $\leg(c)=\lambda_j'-i$, which is the number
of cells strictly below $c$ in its column. We can now define
the following partition statistics:

\begin{itemize}
\item The \emph{size} of $\gamma$ is $|\gamma|=\sum_{i\geq 1} \gamma_i$,
 which is the number of cells in the diagram of $\gamma$.
\item The \emph{diagonal inversion count}
 $\dinv(\gamma)$ is the number of cells $c$ in the diagram of
 $\gamma$ such that $\arm(c)-\leg(c)\in\{0,1\}$.
\item The \emph{deficit} of $\gamma$ is $\defc(\gamma)=|\gamma|-\dinv(\gamma)$,
 which is a nonnegative integer.
\item For each integer $n>0$, the \emph{$n$-triangle}
 $\Delta_n$ is the diagram of the partition $(n-1,n-2,\ldots,3,2,1)$.
\item The \emph{minimum triangle size} of $\gamma$, 
 denoted $\mind(\gamma)$, is the least integer $n$ such that
 $\dg(\gamma)\subseteq \Delta_n$. Equivalently, $\mind(\gamma)$
 is the least integer $n$ such that $\gamma_i\leq n-i$ 
 for $1\leq i\leq\ell(\gamma)$.
 (This statistic was denoted $\Delta(\gamma)$ in~\cite{LLL18}.)
\item For any $n$ such that $\dg(\gamma)\subseteq\Delta_n$, 
 the \emph{external area} of $\gamma$ relative to $\Delta_n$
 is $\area_n(\gamma)=|\Delta_n|-|\gamma|=\binom{n}{2}-|\gamma|$.
 This is the number of cells in the triangle $\Delta_n$ outside
 the diagram of $\gamma$.  
\end{itemize}

\begin{example}\label{ex:ptn-in-tri}
Let $\gamma=(5,4,1,1,1)$. This partition has length $\ell(\gamma)=5$,
size $|\gamma|=12$, diagonal inversion count $\dinv(\gamma)=8$,
deficit $\defc(\gamma)=4$, and minimum triangle size $\mind(\gamma)=6$.
Figure~\ref{fig:ptn-in-tri} shows the diagram of $\gamma$ embedded
in the non-minimal triangle $\Delta_7$. Counting the shaded cells,
we see that the external area of $\gamma$ relative to $\Delta_7$ is 
$\area_7(\gamma)=\binom{7}{2}-|\gamma|=9$, whereas $\area_6(\gamma)=3$.
The eight cells marked with a dot contribute to $\dinv(\gamma)$,
while the other four cells in the diagram of $\gamma$ contribute
to $\defc(\gamma)$. For example, the second cell $c$ in row $1$
contributes to $\defc(\gamma)$ since $\arm(c)=3$ and $\leg(c)=1$,
while the third cell $c'$ in row $1$ contributes to $\dinv(\gamma)$
since $\arm(c')=2$ and $\leg(c')=1$.  
\begin{figure}[h]
\begin{center}
\epsfig{file=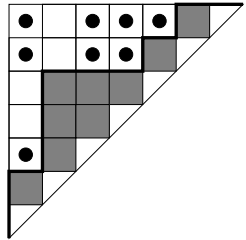,scale=0.7}
\end{center}
\caption{A partition contained in the triangle $\Delta_7$.}
\label{fig:ptn-in-tri}
\end{figure}
\end{example}

Next we define some special collections of integer partitions.

\begin{itemize}
\item Let $\Par$ be the set of all integer partitions.
\item Let $\Par(n)$ be the set of all integer partitions of size $n$.
\item Let $\DP(n)=\{\gamma\in\Par: \dg(\gamma)\subseteq\Delta_n\}$
be the set of all partitions whose diagrams fit in the $n$-triangle.
We call such partitions \emph{Dyck partitions of order $n$}
since these partitions correspond bijectively to Dyck paths of
order $n$ by taking the southeast border of $\gamma$ in $\Delta_n$
(see the thick shaded line in Figure~\ref{fig:ptn-in-tri}).
Observe that 
\begin{equation}\label{eq:DP-mind}
\DP(n)=\{\gamma\in\Par: \mind(\gamma)\leq n\}.
\end{equation}
\item Let $\Def(k)=\{\gamma\in\Par: \defc(\gamma)=k\}$ be 
the set of all partitions having deficit $k$. 
(This set was denoted $\DP_{\ast,k}$ in~\cite{LLL18}.)
\end{itemize}

Continuing Example~\ref{ex:ptn-in-tri}, note that the
partition $\gamma=(5,4,1,1,1)$ is a member of the sets
$\Par(12)$, $\Def(4)$, and $\DP(n)$ for all $n\geq 6$,
since $\mind(\gamma)=6$.

We can now rewrite the definition~\eqref{eq:def-qtcat} of
$q,t$-Catalan numbers as follows:
\[ \Cat_n(q,t)=\sum_{\gamma\in\DP(n)} q^{\area_n(\gamma)}t^{\dinv(\gamma)}. \]
For example, $\DP(3)=\{(0),(1),(2),(1,1),(2,1)\}$, and
\[ \Cat_3(q,t)=q^3+q^2t+qt+qt^2+t^3. \]

\subsection{Dyck Vectors}
\label{subsec:dyck-vec}

For many calculations involving dinv, it is convenient to use 
Dyck vectors instead of Dyck partitions. A \emph{Dyck vector} of order $n$
is a list $v=(v_1,v_2,\ldots,v_n)$ of nonnegative integers such that
$v_1=0$ and $v_{i+1}\leq v_i+1$ for $1\leq i<n$. Let $\DV(n)$ be the
set of Dyck vectors of order $n$. There is a bijective correspondence
between the sets $\DP(n)$ and $\DV(n)$, which can be defined pictorially
by drawing $\gamma\in\DP(n)$ inside $\Delta_n$ and letting $v_i$ be the 
number of external area cells in the $i$th row from the bottom. 
For example, letting $n=7$, the partition $\gamma=(5,4,1,1,1,0,0)$ shown in
Figure~\ref{fig:ptn-in-tri} maps to the Dyck vector $v=(0,1,1,2,3,1,1)$.
Formally, the bijection $\dvmap_n:\DP(n)\rightarrow\DV(n)$ and its
inverse $\dpmap_n:\DV(n)\rightarrow\DP(n)$ are given by these formulas:
\begin{equation}\label{eq:dvmap}
 \DV_n(\gamma_1,\gamma_2,\ldots,\gamma_n) 
 = (0-\gamma_n,1-\gamma_{n-1},\ldots,i-\gamma_{n-i},\ldots,n-1-\gamma_1); 
\end{equation}
\begin{equation}
\label{eq:dpmap}
\DP_n(v_1,v_2,\ldots,v_n)
 = (n-1-v_n,n-2-v_{n-1},\ldots,n-i-v_{n-i+1},\ldots,1-v_2,0-v_1).
\end{equation}

We can compute the statistics $\area_n$, $\dinv$, and $\defc$ directly
from the Dyck vector associated to a Dyck partition. In more detail,
for any Dyck vector $v\in\DV(n)$, define
\begin{align*}
 \area_n(v) &= v_1+v_2+\cdots+v_n; \\
 \dinv(v) &= \text{the number of $i<j$ with $v_i-v_j\in\{0,1\}$; } \\
 \defc(v) &= \binom{n}{2}-\area_n(v)-\dinv(v).
\end{align*}
The bijections defined above preserve all three statistics, so
that if $v=\dvmap_n(\gamma)$, then $\area_n(v)=\area_n(\gamma)$,
$\dinv(v)=\dinv(\gamma)$, and $\defc(v)=\defc(\gamma)$. The verification
of this assertion for dinv is not completely routine --- 
see~\cite[Lemma 4.4.1]{HHLRU} for details.

\subsection{The Successor Map $\nu$}
\label{subsec:map-nu}

This section recalls the definition and properties of 
the \emph{successor map} $\nu$, which is a function 
(defined on a subset of $\Par$) that suffices to 
construct ``almost all'' of the links in the global chains $\C_{\mu}$.
Intuitively, if $\gamma\in\C_{\mu}$ is a partition in the domain of $\nu$,
then $\nu(\gamma)$ is the next partition in the chain $\C_{\mu}$.

The domain of $\nu$ is 
$\{\gamma\in\Par: \gamma_1\leq \ell(\gamma)+2\}$. 
For $\gamma=(\gamma_1,\ldots,\gamma_\ell)$ in this domain, we define
\[ \nu(\gamma)=(\ell(\gamma)+1,\gamma_1-1,\gamma_2-1,\ldots,\gamma_\ell-1). \]
Pictorially, we obtain the diagram of $\nu(\gamma)$ from the diagram
of $\gamma$ by removing the leftmost column, then inserting a new top
row that is one cell longer than the removed column.
For example, $\nu(5,4,1,1,1)=(6,4,3,0,0,0)=(6,4,3)$, 
whereas $\nu(6,4,3)$ is undefined. The key property of $\nu$ 
(proved in Lemma~2.3 of~\cite{LLL18}) is that for all $\gamma$
in the domain of $\nu$,
\[ \dinv(\nu(\gamma))=\dinv(\gamma)+1\quad\mbox{and}\quad
   \defc(\nu(\gamma))=\defc(\gamma). \]
We may also conclude that $\area_n(\nu(\gamma))=\area_n(\gamma)-1$
if $\gamma$ and $\nu(\gamma)$ are both in $\DP(n)$. 

It is readily checked that the image of $\nu$ is the set
$\{\delta\in\Par: \delta_1\geq\ell(\delta)\}$. 
For a partition $\delta=(\delta_1,\ldots,\delta_s)$ in this set, we have
\[ \nu^{-1}(\delta)=(\delta_2+1,\delta_3+1,\ldots,\delta_s+1,
  \underline{1}^{\delta_1-\ell(\delta)}), \]
where the notation $\underline{1}^{\delta_1-\ell(\delta)}$
denotes $\delta_1-\ell(\delta)$ copies of $1$.
We say $\delta\in\Par$ is an \emph{initial} partition if
$\nu^{-1}(\delta)$ is undefined, i.e., $\delta_1<\ell(\delta)$.
We say $\gamma\in\Par$ is a \emph{final} partition if
$\nu(\gamma)$ is undefined, i.e., $\gamma_1>\ell(\gamma)+2$.

Remarkably, for every partition $\mu$, we can build the whole
infinite tail of the chain $\C_{\mu}$ by starting with a particular
partition $\TI(\mu)$ and applying $\nu$ repeatedly. Suppose $\mu$
has $n_1$ parts equal to $1$, $n_2$ parts equal to $2$, and so on.
Taking $N=\mu_1+\ell(\mu)+1$, the \emph{tail initiator partition} 
of type $\mu$ is
\[ \TI(\mu)=\dpmap_N(0,0,\underline{1}^{n_1},0,\underline{1}^{n_2},
   0,\ldots,0,\underline{1}^{n_{\mu_1}}). \]
(This partition was denoted $\gamma_{\mu}$ in~\cite{LLL18}.)
For example, $\mu=(4,3,1,1,1)$ has $n_1=3$, $n_2=0$, $n_3=n_4=1$, 
and $N=10$, so
\[ \TI(\mu)=\dpmap_{10}(0,0,1,1,1,0,0,1,0,1)=(8,8,6,6,5,3,2,1,1,0). \]
For any $\mu$, the Dyck vector associated with $\TI(\mu)$
starts with two $0$s and ends with a $1$, which implies that the length
of $\TI(\mu)$ is $1$ more than the longest part of $\TI(\mu)$.
Thus every $\TI(\mu)$ is an initial partition. Moreover,
it is shown in~\cite[Lemmas 6.7 and 6.8]{LLL18} that $\nu^m(\TI(\mu))$ 
is defined for every integer $m\geq 0$, $\defc(\TI(\mu))=|\mu|$,
and $\dinv(\TI(\mu))=\binom{\mu_1+\ell(\mu)+1}{2}-\ell(\mu)-|\mu|$. 
We call the set
\[ \tail(\mu)=\{\nu^m(\TI(\mu)):m\geq 0\} \]
the \emph{$\nu$-tail of the chain $\C_{\mu}$}.
Note that $\mu$ is uniquely determined by the sequence $\tail(\mu)$, 
as follows.  First, $\TI(\mu)$ is the unique object with minimum dinv 
in $\tail(\mu)$. Second, we can find the multiplicities of the parts of $\mu$
by counting consecutive $1$s in $\dvmap_N(\TI(\mu))$, where $N=\mind(\TI(\mu))$.

Now suppose $\gamma$ is any partition.
The \emph{$\nu$-segment generated by $\gamma$} is the set of partitions
obtained by applying $\nu$ and $\nu^{-1}$ to $\gamma$ as many times as
possible. Formally, the $\nu$-segment $\nu^*(\gamma)$ is the set of
all partitions $\nu^m(\gamma)$ for those integers $m$ such that
$\nu^m(\gamma)$ is defined. All $\nu$-tails are $\nu$-segments.
An example of a finite $\nu$-segment is
\[ \nu^*(5,2,2,2)=\{(3,3,3,1),(5,2,2,2),(5,4,1,1,1),(6,4,3)\}. \] 
Since $\nu$ and $\nu^{-1}$ are one-to-one
on their domains, the set $\Par$ and all of its subsets $\Def(k)$
are disjoint unions of $\nu$-segments.  We hope to express
each chain $\C_{\mu}$ as the union of certain (suitably chosen)
$\nu$-segments, one of which is the $\nu$-tail of $\C_{\mu}$. The hard
part of the construction is figuring out which $\nu$-segments can be
combined to make the needed opposite property hold.

\subsection{The Global Chain Decomposition Conjecture}
\label{subsec:global-conj}

We now have all the ingredients needed for the main structural
conjecture on global chains. The conjecture stated here consists of parts
(a), (b), (c), (d), (f), (g), and (j) of Conjecture~6.9 in~\cite{LLL18}. 
Our conjecture on local chains (given in~\S\ref{subsec:local-chain-conj})
implies this version of the global conjecture.  
Recall from~\eqref{eq:Cat-n-mu} the notation
\begin{equation}\label{eq:CatnS}
 \Cat_{n,\calS}(q,t)=\sum_{\gamma\in\calS\cap\DP(n)}
 q^{\area_n(\gamma)}t^{\dinv(\gamma)}. 
\end{equation}

\begin{conjecture}\label{conj:global}
There exist collections of partitions $\C_{\mu}$,
indexed by deficit partitions $\mu$, and a
size-preserving involution $\mu\mapsto\mu^*$
on $\Par$, satisfying the following conditions:
\begin{enumerate}
\item[(a)] The collections $\C_{\mu}$ are pairwise disjoint.  
\item[(b)] For all $\gamma\in\C_{\mu}$, $\defc(\gamma)=|\mu|$.
\item[(c)] Each $\C_{\mu}$ has the form 
 $\{C_{\mu}(a),C_{\mu}(a+1),C_{\mu}(a+2),\ldots\}$, \\
  where $a=\ell(\mu^*)$ and $\dinv(C_{\mu}(i))=i$ for all $i\geq a$.  
\item[(d)] Every $\gamma\in\Def(k)$ belongs to $\C_{\mu}$ for some
 $\mu\in\Par(k)$.
\item[(e)] For all $\mu$, $\tail(\mu)\subseteq\C_{\mu}$.
\item[(f)] For all $\gamma\in\C_{\mu}$, if $\nu(\gamma)$ is defined
 then $\nu(\gamma)\in\C_{\mu}$.  
\item[(g)] For all $n\geq 0$ and all $\mu$, 
 $\Cat_{n,\C_{\mu}}(q,t)=\Cat_{n,\C_{\mu^*}}(t,q)$.  
\end{enumerate}
\end{conjecture}

Parts (a), (b), and (d) say that each set $\Def(k)$ is the disjoint union
of the chains $\C_{\mu}$ with $\mu\in\Par(k)$. Parts (b) and (c) say that
each chain $\C_{\mu}$ is an infinite string of objects of deficit
$|\mu|$, where dinv increases by $1$ as we move along the string,
and the first object in the string has dinv equal to $\ell(\mu^*)$.  
Part (e) says that the ``right end'' of $\C_{\mu}$ is the
$\nu$-tail $\tail(\mu)$. (More generally, Conjecture~6.9(g) of~\cite{LLL18}
asserts that each $\C_{\mu}$ is closed under $\nu$ and is therefore a union 
of $\nu$-segments, but this cannot be deduced from our local conjecture.) 
Part (f) says that each chain $\C_{\mu}$ is closed under $\nu$
(and hence is closed under $\nu^{-1}$). Part (f) is equivalent to
requiring that for all $\gamma\in\C_{\mu}$, 
the whole $\nu$-segment $\nu^*(\gamma)$ is contained in $\C_{\mu}$.  
Part (g) is the crucial \emph{global opposite property} of the global chains.  
We note that $\mu^*$ is usually not the transpose of $\mu$,
and $\mu^*=\mu$ can occur. For instance, $(k)^*=(k)$ as seen in~\cite{LLL18}.  

\begin{remark}
By invoking a deep result from~\cite{LW}, we can prove that there is
a way to satisfy conditions (a) through (d) of Conjecture~\ref{conj:global}.
Specifically, the main result of~\cite{LW} yields an explicit (but extremely
intricate) bijection $\Phi$ on $\Par$ that preserves size and sends dinv
to the length of the first part:
\[ |\Phi(\gamma)|=|\gamma|\mbox{ and }
   \Phi(\gamma)_1=\dinv(\gamma)\mbox{ for all $\gamma\in\Par$.} \]
For each partition $\mu$, let $\calB_{\mu}$ be the set of partitions
obtained by adding a new longest part to $\mu$ of any size $i\geq a$,
where $a=\mu_1=\ell(\mu')$.  (So $\mu^*$ is $\mu'$ in this remark.)
Let $\C_{\mu}=\{\Phi^{-1}(\gamma):\gamma\in\calB_{\mu}\}$.
By the properties of $\Phi$ cited above, each object $(i,\mu)$ in $\calB_{\mu}$
maps to an object $C_{\mu}(i)$ in $\C_{\mu}$ having dinv equal to $i$
and deficit equal to $|\mu|$. Since $\Par$ is clearly the disjoint union of the
sets $\calB_{\mu}$, $\Par$ is also the disjoint union of the sets $\C_{\mu}$. 
But, we have checked that the chains $\C_{\mu}$ in this remark
do \emph{not} satisfy the opposite condition~\ref{conj:global}(g).  
\end{remark}

\section{Local Chains}
\label{sec:local-chains}

This section develops the theory of local chains.
Section~\ref{subsec:opp-mind} begins by relating the opposite property
of global chains to the sequence of $\mind$ values of objects in the chains.
The next three subsections study staircase sequences and sequences
built from these by a pointwise minimum operation. It turns out that
the needed opposite property has a remarkably simple proof in this
abstract setting. The formal definition of local chains appears
in Sections~\ref{subsec:ord-local-chain} and~\ref{subsec:exc-local-chain}. 
We state our main conjecture on local chains in 
Section~\ref{subsec:local-chain-conj}, and we 
show how Conjecture~\ref{conj:global} follows from this local conjecture.

\subsection{The Opposite Property for $\mind$-Sequences}
\label{subsec:opp-mind}

Suppose $\calS=(\gamma(i): i\geq a)$ is a sequence
of partitions in $\Def(k)$ such that $\dinv(\gamma(i))=i$ for all $i\geq a$.
We have $\defc(\gamma(i))=k$ and $|\gamma(i)|=k+i$ for all $i\geq a$.
To compute the polynomials $\Cat_{n,\calS}(q,t)$, we consider
the function (sequence) $F:\Z_{\geq a}\rightarrow \Z$ defined by 
$F(i)=\mind(\gamma(i))$ for all $i\geq a$.
We call $F$ the \emph{$\mind$-sequence associated with $\calS$}.
By~\eqref{eq:DP-mind} and~\eqref{eq:CatnS}, for any $n\geq 0$ we have
\begin{equation}\label{eq:Cat-from-F}
 \Cat_{n,\calS}(q,t)=\sum_{i:\ F(i)\leq n} q^{\binom{n}{2}-k-i}t^i.
\end{equation}

Next suppose $\calS^*=(\delta(j): j\geq b)$ is another sequence
of partitions in $\Def(k)$ with $\dinv(\delta(j))=j$ for all $j\geq b$.
Let $G:\Z_{\geq b}\rightarrow\Z$ be the associated $\mind$-sequence given by
$G(j)=\mind(\delta(j))$ for all $j\geq b$. For all $n\geq 0$,
$\Cat_{n,\calS^*}(t,q) = \sum_{j:\ G(j)\leq n} t^{\binom{n}{2}-k-j}q^j$. 
Making the change of variable $i=\binom{n}{2}-k-j$, we have
\[ \Cat_{n,\calS^*}(t,q) = \sum_{i:\ G(\binom{n}{2}-k-i)\leq n}
 q^{\binom{n}{2}-k-i}t^{i}. \]
Comparing to~\eqref{eq:Cat-from-F}, we see that the sequences
$\calS$ and $\calS^*$ have the opposite property (namely, 
$\Cat_{n,\calS^*}(t,q) =\Cat_{n,\calS}(q,t)$ for all $n$) 
iff $F$ and $G$ are related by the following condition:
\begin{equation}\label{eq:opposite-fns}
\text{for all integers $n$ and $i$, } 
 F(i)\leq n\Leftrightarrow G\left(\binom{n}{2}-k-i\right)\leq n.
\end{equation}
Here and below, a statement such as ``$G(\binom{n}{2}-k-i)\leq n$''
is an abbreviation for ``$\binom{n}{2}-k-i$ is in the domain of $G$
and $G(\binom{n}{2}-k-i)\leq n$.'' By definition, we say that
any two functions $F$ and $G$ satisfying~\eqref{eq:opposite-fns} 
have the \emph{opposite property for deficit $k$}.

\subsection{Staircase Sequences}
\label{subsec:staircase-seq}

We intend to build functions ($\mind$-sequences) having the 
opposite property~\eqref{eq:opposite-fns} by taking
the pointwise minimum of certain other sequences with special structure.
These latter sequences are defined as follows.

\begin{definition}\label{def:staircase}
Given integers $a,m,h$, the \emph{infinite $(a,m,h)$-staircase}
is the  function (sequence) $F:\Z_{\geq a}\rightarrow \Z$ such that
the values $(F(a),F(a+1),F(a+2),\ldots)$ consist of 
$m+1$ copies of $h$,
followed by $h$ copies of $h+1$, 
followed by $h+1$ copies of $h+2$, 
followed by $h+2$ copies of $h+3$, and so on.
A \emph{finite $(a,m,h)$-staircase} is any finite prefix of the
sequence $F$, that is, a function obtained by restricting $F$
to a domain of the form $\{a,a+1,\ldots,b\}$.
\end{definition}

Here is an explicit formula for the values of the infinite
$(a,m,h)$-staircase $F$: 
\[ \begin{array}{lcl}
F(i)=h &\mbox{ if }& a\leq i\leq a+m; \\
F(i)=h+1 &\mbox{ if }& a+m+1\leq i\leq a+m+h; \\
F(i)=h+2 &\mbox{ if }& a+m+h+1\leq i\leq a+m+h+(h+1); \\
F(i)=h+3 &\mbox{ if }& a+m+h+(h+1)+1\leq i\leq a+m+h+(h+1)+(h+2); \\
\ldots & & \ldots \\
F(i)=h+p &\mbox{ if }& a+m+h+(h+1)+\cdots + (h+p-2)+1 \leq i
  \\ & & \qquad \leq a+m+h+(h+1)+\cdots+(h+p-1).
\end{array} \] 
So (taking $h+p=n$ above) 
\begin{equation}\label{eq:stair-val}
\mbox{for all $n>h$, }
 F(i)=n \ \Leftrightarrow\ a+m+\binom{n-1}{2}-\binom{h}{2} < i
  \leq a+m+\binom{n}{2}-\binom{h}{2}.  
\end{equation}

The following result shows that staircase sequences arise naturally
by taking the $\mind$-sequence associated with a $\nu$-segment or $\nu$-tail.
The special case of a $\nu$-tail was proved in~\cite[Lemma 6.8(2)]{LLL18}.

\begin{proposition}\label{prop:nu-stair}
Suppose $\gamma$ is a partition with $\dinv(\gamma)=a$, $\mind(\gamma)=n$,
and $m\geq 0$ is the least integer such that $\gamma_{m+1}'=n-m-1$.
Let $I$ be the set of $i\geq 0$ such that $\nu^i(\gamma)$ is defined. 
Then the sequence $F$ with domain $\{a+i: i\in I\}$ given by
$F(a+i)=\mind(\nu^i(\gamma))$ is an $(a,m,n)$-staircase.  
\end{proposition}
\begin{proof}
We first note that the value $m$ has the following geometric interpretation.
Draw the diagram of $\gamma$ inside the minimal triangle $\Delta_n$ 
and look for the lowest point where a cell in the diagram touches the 
diagonal boundary $y=x$ of $\Delta_n$. This point has coordinates $(m+1,m+1)$,
as is readily checked. We call this point the \emph{first-return point}
for $\gamma$ (relative to $\Delta_n$).  

Suppose $m\geq 1$ and $\nu(\gamma)$ is defined. As noted earlier,
we obtain the diagram of $\nu(\gamma)$ by removing the leftmost column
of $\dg(\gamma)$ and making a new top row that is $1$ cell longer. Since
the diagram of $\gamma$ does not touch $(1,1)$, the new diagram still
fits in $\Delta_n$ with first-return point $(m,m)$. Thus, 
$\mind(\nu(\gamma))=n$.  Similarly, if $m\geq 2$ and $\nu^2(\gamma)$
is defined, then the diagram of $\nu^2(\gamma)$ still fits in
$\Delta_n$ with first-return point $(m-1,m-1)$. For any $m\geq 0$,
we can apply this reasoning for $m$ steps (always assuming the relevant 
powers $\nu^i(\gamma)$ are defined), until we eventually obtain the 
diagram of $\nu^m(\gamma)$ inside $\Delta_n$ with first-return point $(1,1)$.
Thus the first $m+1$ values of $F$ are $n$, as needed.
See Figure~\ref{fig:mind-seq} for an example where
$\gamma=(5,5,3,3,1)$, $a=14$, $n=7$, $m=2$, and the $F$-sequence
starts $(\underline{7}^3,\underline{8}^7,\underline{9}^8,
 \underline{10}^9,\ldots)$.

\begin{figure}[h]
\begin{center}
\epsfig{file=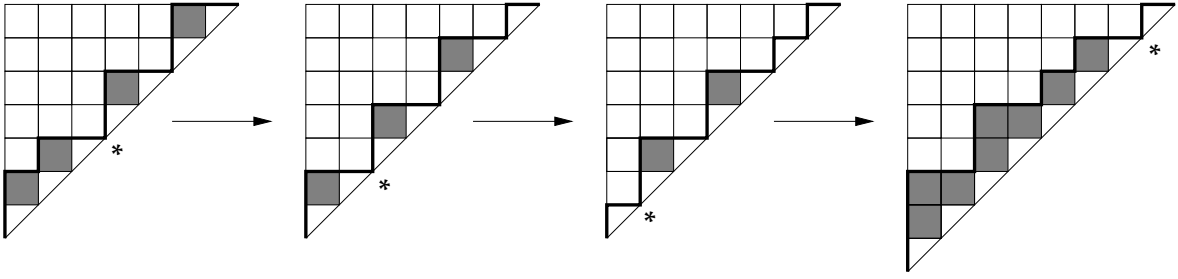,scale=0.7}
\end{center}
\caption{Finding the minimum triangle sizes for the sequence
 $(\nu^i(\gamma): i\geq 0)$. The first-return point
 for each object is starred.}
\label{fig:mind-seq}
\end{figure}

Now consider $\nu^{m+1}(\gamma)$ (if defined). Here we remove the
first column of size $n-1$ and add a new first row of size $n$. 
This new row no longer fits inside $\Delta_n$, but it does fit
inside $\Delta_{n+1}$. We conclude that $\mind(\nu^{m+1}(\gamma))=n+1$,
and the first-return point is now $(n,n)$. Repeating the reasoning
in the previous paragraph, we see that the next $n$ partitions
$\nu^{m+1}(\gamma),\ldots,\nu^{m+n}(\gamma)$ (if defined) will
all have minimum triangle size $n+1$, as the first-return point
moves from $(n,n)$ to $(1,1)$ one step at a time. After that,
the next $n+1$ partitions will have minimum triangle size $n+2$,
with first-return point moving from $(n+1,n+1)$ to $(1,1)$. This
reasoning can be continued forever, unless we eventually reach
a final partition where $\nu$ is undefined. In either case, we
have proved that the sequence $F$ is an $(a,m,n)$-staircase.  
\end{proof}

\subsection{The Pointwise Minimum of Staircase Sequences}
\label{subsec:min-of-stair}

Suppose $F_1,\ldots,F_s$ are integer-valued sequences with respective
domains $D_1,\ldots,D_s\subseteq\Z$. The \emph{pointwise minimum}
of $F_1,\ldots,F_s$ is the sequence $F$ with domain $D=D_1\cup\cdots\cup D_s$
such that for each $j\in D$, 
$F(j)$ is the least integer in the set $\{F_i(j): j\in D_i\}$.
We write $F=\min(F_1,\ldots,F_s)$. 
The next lemma gives an explicit description of the sets
$\{ i: F(i)\leq n \}$ in the case where $F$ is the pointwise minimum
of staircase sequences. For any $a,b\in\Z$, we use the notation 
$[a,b]=\{a,a+1,\ldots,b\}$ for an interval of consecutive integers. 
This interval is the empty set if $a>b$.

\begin{lemma}\label{lem:min-stair}
Let $F_j$ be the infinite $(a_j,m_j,h_j)$-staircase for $1\leq j\leq s$,
and let $F=\min(F_1,\ldots,F_s)$.  For every integer $n\geq 0$,
\begin{equation}\label{eq:lem-min-stair}
 \{ i: F(i)\leq n \} = 
 \bigcup_{j:\ n\geq h_j} \left[a_j, a_j+m_j+\binom{n}{2}-\binom{h_j}{2}\right].
\end{equation}
\end{lemma}
\begin{proof}
By~\eqref{eq:stair-val}, for $1\leq j\leq s$ and any $n\geq h_j$,
\[ \{ i: F_j(i)\leq n \} 
 = \left[a_j, a_j+m_j+\binom{n}{2}-\binom{h_j}{2}\right]. \] 
On the other hand, if $n<h_j$, then $\{ i: F_j(i)\leq n \}=\emptyset$.
Fix $n\geq 0$ and $i$ in the domain of $F$. By definition of pointwise
minimum and since $F_j$ has minimum value $h_j$,
\begin{align*}
F(i)\leq n &\Leftrightarrow
 \text{for some $j$, $F_j(i)\leq n$} 
\\ &\Leftrightarrow
 \text{for some $j$, $n\geq h_j$ and $F_j(i)\leq n$}
\\ &\Leftrightarrow
 i\in \bigcup_{j:\ n\geq h_j} 
   \left[a_j, a_j+m_j+\binom{n}{2}-\binom{h_j}{2}\right]. \qedhere
\end{align*}
\end{proof}

\subsection{The Opposite Property for Minima of Staircases}
\label{subsec:opp-prop-lemma}

\begin{lemma}\label{lem:LG}
Fix a deficit value $k\geq 0$, an integer $N\geq 2$,
and nonnegative integers $a_j,b_j,m_j,h_j$ for $0<j<N$.  
For $0<j<N$, let $F_j$ be the infinite $(a_j,m_j,h_j)$-staircase,
let $G_j$ be the infinite $(b_{N-j},m_{N-j},h_{N-j})$-staircase,
let $F=\min(F_j)$, and let $G=\min(G_j)$. Assume that
\begin{equation}\label{eq:LG-hyp}
 a_j+b_j+m_j+k=\binom{h_j}{2}\mbox{ for $0<j<N$.}
\end{equation} 
Then $F$ and $G$ have the opposite property~\eqref{eq:opposite-fns}
 for deficit $k$.
\end{lemma}
\begin{proof}
Fix integers $n$ and $i$. By Lemma~\ref{lem:min-stair}
and the assumption $a_j+m_j-\binom{h_j}{2}=-b_j-k$,
\[ F(i)\leq n \Leftrightarrow
 i \in \bigcup_{j:\ n\geq h_j} 
       \left[a_j,\binom{n}{2}-b_j-k\right]
 \Leftrightarrow \exists j,n\geq h_j\mbox{ and }
  a_j\leq i\leq \binom{n}{2}-b_j-k. \]
Similarly, applying Lemma~\ref{lem:min-stair} to $G$ 
and replacing $j$ by $N-j$, we get:
\begin{align*}
G\left(\binom{n}{2}-k-i\right)\leq n &\Leftrightarrow
 \binom{n}{2}-k-i \in \bigcup_{j:\ n\geq h_{N-j}} 
       \left[b_{N-j},b_{N-j}+m_{N-j}+\binom{n}{2}-\binom{h_{N-j}}{2}\right] \\
& \Leftrightarrow  \binom{n}{2}-k-i \in \bigcup_{j:\ n\geq h_j} 
       \left[b_j,b_j+m_{j}+\binom{n}{2}-\binom{h_{j}}{2}\right] \\
&\Leftrightarrow  \exists j,n\geq h_j\mbox{ and }
  b_j\leq \binom{n}{2}-k-i \leq \binom{n}{2}-k-a_j \\
&\Leftrightarrow \exists j,n\geq h_j\mbox{ and }
  a_j\leq i\leq \binom{n}{2}-b_j-k
 \Leftrightarrow F(i)\leq n.  \qedhere
\end{align*}
\end{proof}

The next two lemmas consider a special situation where we can
conclude $F_1\geq F_2\geq\cdots\geq F_{N-1}$; this situation will
arise in our study of local chains.

\begin{lemma}\label{lem:compare2}
Suppose $F$ is the infinite $(a,m,h)$-staircase,
        $G$ is the infinite $(a',m',h')$-staircase, 
$a+m<a'$, and $G(a')<F(a')$. Then for all $i\geq a'$, $G(i)\leq F(i)$.
\end{lemma}
\begin{proof}
We know $F$ and $G$ are weakly increasing sequences whose values
increase by $0$ or $1$ at each step.
Let $y=F(a')$, let $i_1$ be the least integer with $F(i_1)=y$,
and let $i_2$ be the least integer with $G(i_2)=y$. 
Since $a+m<a'$, we know $F(a)=\cdots=F(a+m)<F(a+m+1)\leq y$,
so $a+m<i_1\leq a'$.  Since $G(a')<y$, the definition of a staircase 
sequence shows that $i_2$ exists and $a'<i_2$.
The values of the sequence $F$, from input $i_1$ onward,
 are $y-1$ copies of $y$, then $y$ copies of $y+1$, and so on.
The values of the sequence $G$, from input $i_2$ onward,
 are $y-1$ copies of $y$, then $y$ copies of $y+1$, and so on.
If $i_2\leq i$, then these remarks show that
\[ G(i)=F(i_1+(i-i_2))=F(i+(i_1-i_2))\leq F(i), \]
since $i_1-i_2<0$. If $a'\leq i<i_2$, 
then $G(i)\leq G(i_2)=y=F(i_1)\leq F(a')\leq F(i)$.
\end{proof}

\begin{lemma}\label{lem:dec-min}
Suppose $F_j$ is the infinite $(a_j,m_j,h_j)$-staircase for $1\leq j\leq c$,
$a_{j-1}+m_{j-1}<a_j$ and $F_j(a_j)<F_{j-1}(a_j)$ for $1<j\leq c$, and
$F=\min_{1\leq j\leq c}(F_j)$. Let $a_{c+1}=\infty$. If $a_j\leq i<a_{j+1}$, 
then 
\begin{equation}\label{eq:dec-min}
 F(i)=\min\{ F_1(i),F_2(i),\ldots,F_j(i) \}=F_j(i).
\end{equation}
\end{lemma}
\begin{proof} 
The domain of $F_j$ is $\Z_{\geq a_j}$. If $a_j\leq i<a_{j+1}$,
then $i$ is in the domain of $F_1,F_2,\ldots,F_j$ but not
in the domain of $F_{j+1},\ldots,F_c$. So the first equality 
in~\eqref{eq:dec-min} follows from the definition of pointwise minimum.
To get the second equality, we prove the following stronger statement
by induction on $j$: for all $p<j$ and all $i\geq a_j$, $F_j(i)\leq F_p(i)$.
Fix $j>1$, and assume that for all $p<j-1$ and all $i\geq a_{j-1}$,
$F_{j-1}(i)\leq F_p(i)$. Fix $i\geq a_j$. Since $a_{j-1}+m_{j-1}<a_j$
and $F_j(a_j)<F_{j-1}(a_j)$, Lemma~\ref{lem:compare2} shows that
$F_j(i)\leq F_{j-1}(i)$. Combining this with the induction hypothesis,
we see that $F_j(i)\leq F_p(i)$ for all $p<j$, as needed.  
\end{proof}

\subsection{Ordinary Local Chains}
\label{subsec:ord-local-chain}

We are now ready to define local chains.
A sequence of partitions $\calS$ is called an
\emph{ordinary local chain of deficit $k$} iff there exist nonnegative 
integers $a,a',m,m',h,h'$ satisfying the following conditions.
First, $\calS=(\gamma(i): a\leq i\leq a'+m')$ where 
$\defc(\gamma(i))=k$ and $\dinv(\gamma(i))=i$ for $a\leq i\leq a'+m'$.
Second, the $\mind$-sequence $F=(\mind(\gamma(i)): a\leq i\leq a'+m')$ 
associated with $\calS$ satisfies 
\[ F(a'-1)>F(a')=F(a'+1)=\cdots=F(a'+m')=h'. \]
Third, $a+m+1<a'$ and the restriction of $F$ to $\{a,a+1,\ldots,a'-1\}$
is an $(a,m,h)$-staircase, so in particular
\[ h=F(a)=F(a+1)=\cdots=F(a+m)<F(a+m+1). \]

Since $F(a+m)<F(a+m+1)$ and $F(a'-1)>F(a')$, 
the integers $a,a',m,m',h,h'$ are uniquely determined by $\calS$,
and we denote them $a_{\calS}$, $a'_{\calS}$, $m_{\calS}$, $m'_{\calS}$,
$h_{\calS}$, $h'_{\calS}$ (respectively). We also define 
the \emph{left part}, \emph{middle part}, and \emph{right part} of $\calS$ 
to be
\begin{align*}
\lpart(\calS) &= \{\gamma(j):a\leq j\leq a+m\}; \\
\mpart(\calS) &= \{\gamma(j): a+m<j<a'\};  \\
\rpart(\calS) &= \{\gamma(j): a'\leq j\leq a'+m'\}.
\end{align*}

It is helpful to visualize the conditions on the $\mind$-sequence
of a local chain $\calS$ by graphing the set of ordered pairs
$(\dinv(\gamma),\mind(\gamma):\gamma\in\calS)$ in the $xy$-plane.
See Figure~\ref{fig:lc-graph} for an illustration of the structure
of an ordinary local chain.

\begin{figure}[h]
\begin{center}
\epsfig{file=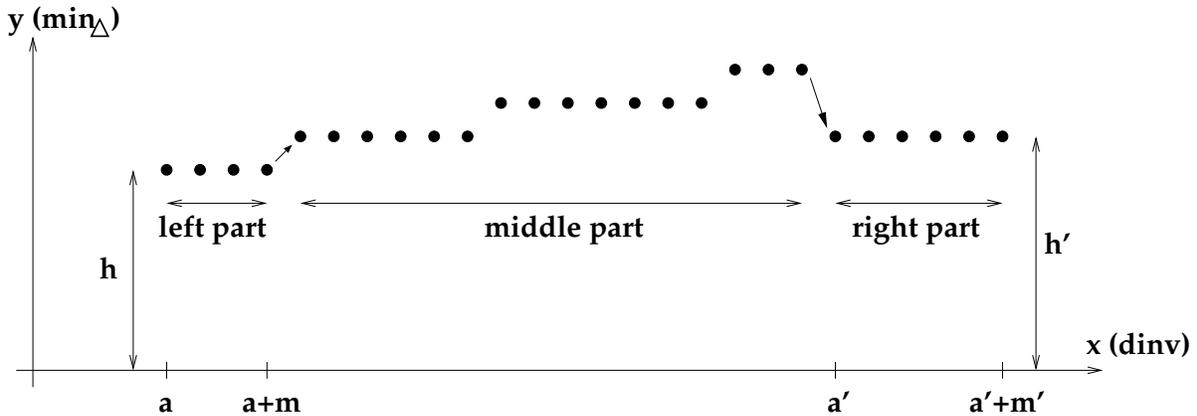,scale=0.7}
\end{center}
\caption{Structure of an ordinary local chain.}
\label{fig:lc-graph}
\end{figure}

\begin{example}\label{ex:ord-local-chain}
Here are two ordinary local chains of deficit $7$:
\begin{align*}
 \calS &= ((3333),(52222),(641111),(753),(44422),(633311)); \\
 \calT &= ((32222),(621111),(751),(43331),(63222),(652111)).
\end{align*}
The following two-line arrays show the values of $\dinv$ and $\mind$
for objects in the sequences $\calS$ and $\calT$:
\[ F_{\calS} = \left[\begin{array}{rrrrrrrr}
 \dinv: & 5 & 6 & 7 & 8 & 9 & 10 \\
 \mind: & 7 & 7 & 7 & 8 & 7 & 7 \end{array}\right]; \qquad
 F_{\calT} = \left[\begin{array}{rrrrrrrr}
 \dinv: & 4 & 5 & 6 & 7 & 8 & 9 \\
 \mind: & 7 & 7 & 8 & 7 & 7 & 7 \end{array}\right].  \]
We have 
\[ a_{\calS}=5,\ m_{\calS}=2,\ h_{\calS}=7,\ 
a'_{\calS}=9,\ m'_{\calS}=1,\ h'_{\calS}=7; \]
\[ a_{\calT}=4,\ m_{\calT}=1,\ h_{\calT}=7,\ 
a'_{\calT}=7,\ m'_{\calT}=2,\ h'_{\calT}=7.  \] 
The left, middle, and right parts of $\calS$ have size $m_{\calS}+1=3$, 
$1$, and $m'_{\calS}+1=2$, respectively.  
\end{example}

Suppose $\calS$ and $\calT$ are ordinary local chains of deficit $k$.
We say that $\calS$ is \emph{locally opposite} to $\calT$ iff
$m_{\calT}=m'_{\calS}$, $m'_{\calT}=m_{\calS}$, 
$h_{\calT}=h'_{\calS}$, $h'_{\calT}=h_{\calS}$, 
\begin{equation}\label{eq:loc-opp}
 a_{\calS}+m_{\calS}+k+a'_{\calT}=\binom{h_{\calS}}{2},\quad\mbox{ and }\quad
 a'_{\calS}+m'_{\calS}+k+a_{\calT}=\binom{h'_{\calS}}{2}. 
\end{equation}
For instance, the local chains $\calS$ and $\calT$ in 
Example~\ref{ex:ord-local-chain} are locally opposite because
$m_{\calT}=1=m'_{\calS}$, $m'_{\calT}=2=m_{\calS}$,
$h_{\calT}=7=h'_{\calS}$, $h'_{\calT}=7=h_{\calS}$, and
(recalling $k=7$)
\[ 5+2+7+7=21=\binom{7}{2},\ \  9+1+7+4=21=\binom{7}{2}. \]

\subsection{Exceptional Local Chains}
\label{subsec:exc-local-chain}

We also need two types of \emph{exceptional local chains of deficit $k$}.
First, any one-element set $\calT=\{\gamma\}$ with $\defc(\gamma)=k$
is an exceptional local chain, and we define
\[ a'_{\calT}=\dinv(\gamma),\ m'_{\calT}=0,\ h'_{\calT}=\mind(\gamma),\]
\[ \lpart(\calT)=\emptyset,\ \mpart(\calT)=\emptyset,\mbox{ and }
   \rpart(\calT)=\{\gamma\}. \]
Second, for any $\mu\in\Par(k)$, the $\nu$-tail 
$\calS=\tail(\mu)=\nu^*(\TI(\mu))$ is an exceptional local chain,
and we define
\[ a_{\calS}=\dinv(\TI(\mu)),\ m_{\calS}=0,\ h_{\calS}=\mind(\TI(\mu)), \]
\[ \lpart(\calS)=\{\TI(\mu)\},\  \mpart(\calS)=\calS\setminus\{\TI(\mu)\},
 \mbox{ and }\rpart(\calS)=\emptyset. \]
By Proposition~\ref{prop:nu-stair} applied to $\gamma=\TI(\mu)$, 
the $\mind$-sequence associated with $\calS=\tail(\mu)$ is the 
infinite $(a_{\calS},m_{\calS},h_{\calS})$-staircase.

Two exceptional local chains of deficit $k$ are \emph{locally opposite} iff
one chain is $\calS=\tail(\mu)$ and the other chain is $\calT=\{\gamma\}$
where $\dinv(\gamma)=\ell(\mu)$ and $\mind(\gamma)=\mind(\TI(\mu))$.
Since we know $|\mu|=k$, $\mind(\TI(\mu))=\mu_1+\ell(\mu)+1$,
and $\dinv(\TI(\mu))=\binom{\mu_1+\ell(\mu)+1}{2}-\ell(\mu)-|\mu|$
(see~\S\ref{subsec:map-nu}), it follows that $m'_{\calT}=0=m_{\calS}$,
$h'_{\calT}=\mind(\gamma)=h_{\calS}$, and the 
first equation in~\eqref{eq:loc-opp} holds.
It follows from these definitions that the relation
``$\calS$ is locally opposite to $\calT$'' is a symmetric
relation on the set of all (ordinary and exceptional)
local chains of deficit $k$.

\begin{example}\label{ex:exc-local-chain}
Here are two exceptional local chains of deficit $6$:
\[ \calS=\tail(411)=\nu^*(6654211)=((6654211),(855431),(774432),\ldots);
\quad \calT=\{(3111111)\}. \]
Writing $\mu=(411)$, $\TI(\mu)=(6654211)$, and $\gamma=(3111111)$, we compute
$\dinv(\gamma)=3=\ell(\mu)$ and $\mind(\gamma)=8=\mind(\TI(\mu))$. 
So $\calS$ and $\calT$ are locally opposite chains. Note that 
\[ a_{\calS}=\dinv(\TI(\mu))=19,\ m_{\calS}=0,\ h_{\calS}=8,\ 
   a'_{\calT}=3, m'_{\calT}=0, h'_{\calT}=8, \]
\[\mbox{ and }a_{\calS}+m_{\calS}+k+a'_{\calT}=19+0+6+3=28
=\binom{h_{\calS}}{2}. \] 
\end{example}

\subsection{The Local Chain Conjecture}
\label{subsec:local-chain-conj}

We can now state our main structural conjecture on local chains.

\begin{conjecture}\label{conj:local}
For every $k\geq 0$, there is a set $\LL$ of local chains of deficit $k$,
and there is an involution $\calS\mapsto\calS^*$ on $\LL$,
satisfying the following conditions:
\begin{enumerate} 
\item[(a)] For every $\mu\in\Par(k)$, $\tail(\mu)$ belongs to $\LL$.
\item[(b)] For any two distinct chains in $\LL$, either the two chains
 are disjoint or the right part of one chain equals the left part
 of the other chain.
\item[(c)] Every $\gamma$ in $\Def(k)$ belongs to exactly one or two
 local chains in $\LL$. In the former case, $\gamma$ belongs to the
 middle part of the chain. In the latter case, $\gamma$ belongs to 
 the right part of one chain and the left part of the other.
\item[(d)] For all $\calS$ in $\LL$, the local chains 
 $\calS$ and $\calS^*$ are locally opposite. 
\item[(e)] For all $\calS,\calT\in\LL$, if $\calS$ and $\calT$ have
 nonempty intersection, then $\calS^*$ and $\calT^*$ have nonempty
 intersection.  
\item[(f)] For all $\calS\in\LL$, $\lpart(\calS)\cup \mpart(\calS)$
is a union of $\nu$-segments.
\end{enumerate} 
\end{conjecture}

Aided by computer searches, we can explicitly construct local chains
proving this conjecture for all $k\leq \maxk$. 
The details appear in Section~\ref{sec:specific-chains}.
First we prove that this conjecture for local chains implies
the corresponding conjecture for global chains.

\begin{theorem}\label{thm:LG}
Conjecture~\ref{conj:local} implies Conjecture~\ref{conj:global}.
\end{theorem}
\begin{proof}
Assume Conjecture~\ref{conj:local} holds for a fixed deficit value $k\geq 0$.
We prove the conclusions of Conjecture~\ref{conj:global} for partitions
 $\mu$ of $k$.

\emph{Step 1.} We construct the global chains $\C_{\mu}$ for $\mu\in\Par(k)$.
Start with the local chain $\calS=\tail(\mu)$, which belongs to $\LL$.
By~\ref{conj:local}(b) and (c), there is a unique $\calS'\in\LL$
with $\rpart(\calS')=\lpart(\calS)=\{\TI(\mu)\}$. If this chain $\calS'$
is not a singleton, then there is a unique $\calS''\in\LL$
with $\rpart(\calS'')=\lpart(\calS')$. We continue to paste together
overlapping local chains in this way until eventually terminating at
an exceptional (singleton) chain. Such a chain must be reached in finitely
many steps, since the minimum dinv value for each local chain strictly
decreases as we proceed. At the end, we have
\begin{equation}\label{eq:Cmu}
 \C_{\mu}=\calS_0\cup\calS_1\cup\cdots\cup \calS_c, 
\end{equation}
where $\calS_0=\{\gamma\}$, $\calS_c=\tail(\mu)$, 
and $\rpart(\calS_j)=\lpart(\calS_{j+1})$ for $0\leq j<c$.
The last condition shows that $\calS_{j+1}$ is uniquely
determined by $\calS_j$. Iterating this, we see that
$\tail(\mu)$ and hence $\mu$ are uniquely determined
by any $\calS_i$ in~\eqref{eq:Cmu}. So two chains constructed 
in this way from two different partitions $\mu$ must be disjoint.
So far, we have built chains $\C_{\mu}$ (for $\mu\in\Par(k)$)
satisfying conditions~\ref{conj:global}(a), (b), (c), and (e),
except for the claim $a=\ell(\mu^*)$ that will be proved later.
Since $a$ is the minimum value of dinv among all objects
in $\C_{\mu}$, we see from the construction
that $a=\dinv(\gamma)$ for the unique $\gamma$ in $\calS_0$.

\emph{Step 2.} We prove two formulas expressing the $\mind$-sequence $F_{\mu}$ 
of $\C_{\mu}$ as the pointwise minimum of staircase sequences 
determined by the chains $\calS_0,\calS_1,\ldots,\calS_c$ in~\eqref{eq:Cmu}. 
For brevity, write $a_j=a_{\calS_j}$ for $0\leq j\leq c$, and 
define $m_j$, $h_j$, $a'_j$, $m'_j$, and $h'_j$ similarly.
Let $F_j$ denote the infinite $(a_j,m_j,h_j)$-staircase, and
let $F'_j$ denote the infinite $(a'_j,m'_j,h'_j)$-staircase. We claim
\begin{equation}\label{eq:min-formula}
 \min_{0<j\leq c} F_j=F_{\mu}=\min_{0\leq j<c} F'_j.
\end{equation}
Since $\rpart(S_{j-1})=\lpart(S_{j})$ for $0<j\leq c$, we must have
$a_{j-1}'=a_{j}$, $m_{j-1}'=m_{j}$, $h_{j-1}'=h_{j}$, and hence
$F_{j-1}'=F_{j}$ for $0<j\leq c$ (compare to Figure~\ref{fig:lc-graph}). 
So the second equality in~\eqref{eq:min-formula} follows from the first
one. We prove the first equality with the help of Lemma~\ref{lem:dec-min}.
Let $F=\min(F_1,\ldots,F_c)$.
By definition of local chains, $a_{j-1}+m_{j-1}<a'_{j-1}=a_j$ and 
\[ F_j(a_j)=h_j=h'_{j-1}<F_{j-1}(a'_{j-1}-1)\leq F_{j-1}(a'_{j-1}) 
   =F_{j-1}(a_j) \]
for $1<j\leq c$. So for $i$ in the range $a_j\leq i<a_{j+1}$
(taking $a_{c+1}=\infty$), the lemma tells us that $F(i)=F_j(i)$. 
Thus it suffices to show that $F_{\mu}(i)=F_j(i)$ for all $i$ in this range.

If $j<c$ and $i$ satisfies $a_j\leq i<a_{j+1}=a'_j$, then the unique object
$C_{\mu,i}$ in $\C_{\mu}$ with $\dinv(C_{\mu,i})=i$ belongs to the
left part or middle part of $\calS_j$. Then $F_{\mu}(i)=\mind(C_{\mu,i})
=F_j(i)$ by the third condition in the definition of an ordinary local chain.
On the other hand, if $a_c\leq i$, then $F_{\mu}(i)=F_c(i)$ because
the $\mind$-sequence of the $\nu$-tail $\calS_c$ is known to be the 
infinite $(a_c,m_c,h_c)$-staircase $F_c$. 

\emph{Step 3.} We construct the involution $\mu\mapsto\mu^*$ on $\Par(k)$
and verify the opposite property~\ref{conj:global}(g). Fix $\mu\in\Par(k)$ and 
consider the local chains $\calS_0,\calS_1,\ldots,\calS_c$ in~\eqref{eq:Cmu}.
Using~\ref{conj:local}(d), let $\calS_0^*,\calS_1^*,\ldots,\calS_c^*$ be the
corresponding chains in $\LL$ such that $\calS_j$ and $\calS_j^*$ are
locally opposite for $0\leq j\leq c$. The exceptional chain $\calS_0=\{\gamma\}$
must be locally opposite to some $\nu$-tail, so that 
$\calS_0^*=\tail(\mu^*)$ for some partition $\mu^*$ of $k$.
Using~\ref{conj:local}(e), since $\calS_0$ and $\calS_1$ overlap,
$\calS_0^*$ and $\calS_1^*$ must overlap as well, and in fact
$\rpart(\calS_1^*)=\lpart(\calS_0^*)$. Similarly, since $\calS_1$
and $\calS_2$ overlap, $\calS_1^*$ and $\calS_2^*$ must also overlap,
with $\rpart(\calS_2^*)=\lpart(\calS_1^*)$. We can continue this reasoning
until reaching $\calS_c^*$, which is locally opposite to $\calS_c=\tail(\mu)$
and must therefore be a singleton chain. We conclude that
the decomposition~\eqref{eq:Cmu} for the chain $\C_{\mu^*}$ looks like
\begin{equation}\label{eq:Cmu*}
 \C_{\mu^*}=\calS_c^*\cup\calS_{c-1}^*\cup\cdots\cup \calS_1^*\cup\calS_0^*.
\end{equation}
We now apply Lemma~\ref{lem:LG} with $N=c+1$, $a_j=a_{\calS_j}$, 
$b_j=a'_{\calS^*_j}$, $m_j=m_{\calS_j}=m'_{\calS^*_j}$, and 
$h_j=h_{\calS_j}=h'_{\calS^*_j}$ for $0<j<N$.
Define $F_j$, $G_j$, $F$, and $G$ as in the lemma.
The first equality in~\eqref{eq:min-formula} shows that $F=F_{\mu}$.
The second equality in~\eqref{eq:min-formula} 
(applied to $\mu^*$, and keeping in mind the reversal of the index
order in~\eqref{eq:Cmu*}) shows that $G=F_{\mu^*}$.
Since $\calS_j$ and $\calS^*_j$ are locally opposite chains,
the lemma hypothesis~\eqref{eq:LG-hyp} follows from
the first equation in~\eqref{eq:loc-opp}, which holds even if $j=c$.
Therefore, the lemma applies to show that $F_{\mu}$
and $F_{\mu^*}$ have the opposite property for deficit $k$.
As seen in~\S\ref{subsec:opp-mind},
this property implies $\Cat_{n,\C_{\mu^*}}(t,q)=\Cat_{n,\C_{\mu}}(q,t)$.

\emph{Step 4.} We prove $a=\ell(\mu^*)$ in~\ref{conj:global}(c)
and deduce~\ref{conj:global}(d). With the above notation, we know
$\calS_0=\{\gamma\}$ is locally opposite to $\calS_0^*=\tail(\mu^*)$,  
so $\dinv(\gamma)=\ell(\mu^*)$ by definition. We saw that
$a=\dinv(\gamma)$ at the end of Step~1. Finally, Lemma 6.11 of~\cite{LLL18} 
proves that part (d) of Conjecture~\ref{conj:global} follows automatically
from parts (a), (b), and (c), which are already known.

\emph{Step 5.} We use~\ref{conj:local}(f) to prove~\ref{conj:global}(f).
By construction, each global chain $\C_{\mu}$ is an overlapping union
of certain local chains $\calS\in\LL$. 
By~\ref{conj:local}(b) and (c), we can also
regard each global chain $\C_{\mu}$ as the disjoint union of
the left and middle parts of these same local chains.  
Since each set $\lpart(\calS)\cup\mpart(\calS)$ 
is a union of $\nu$-segments by hypothesis,
so is the global chain $\C_{\mu}$. This proves that $\C_{\mu}$ is
closed under $\nu$ (and $\nu^{-1}$).
\end{proof}

\section{Global and Local Chains for $|\mu|\leq \maxk$} 
\label{sec:specific-chains}

This section presents specific chains $\C_{\mu}$ satisfying
Conjectures~\ref{conj:local} and~\ref{conj:global} for all
deficit partitions $\mu$ of size at most \maxk. 
First we show that any putative global chain $\C_{\mu}$ can be
decomposed into an overlapping union of local chains in at most one way.
We will see that the local opposite property of the local
chains comprising $\C_{\mu}$ and $\C_{\mu^*}$ can be checked quite easily,
in contrast to the global opposite property from~\cite{LLL18}.
We give an example of this process by presenting the complete verification
for deficit partitions $\mu$ of size $4$. The appendix to~\cite{LLL18}
lists specific global chains $\C_{\mu}$ that happen to satisfy the new 
local conjecture for all $\mu$ with $0\leq |\mu|\leq 6$. So we do not
repeat that data here. However, for deficit values larger than $6$,
some new chains are needed.  We list these chains (and the data needed to 
verify the local opposite property) in the appendix at the end of 
this section. This proves the joint symmetry of the terms in $\Cat_n(q,t)$
of degree $\binom{n}{2}-k$ for all $n\geq 0$ and all $k\leq \maxk$.

\subsection{Decomposing Global Chains into Local Chains}
\label{subsec:global-to-local}

A given global chain $\C_{\mu}$ is an infinite sequence
$(\gamma(i):i\geq d)$ of Dyck partitions of deficit $k$.
A convenient way to present such a chain is by specifying
the initial partitions of the $\nu$-segments comprising $\C_{\mu}$.
This is a finite list that ends with the partition $\TI(\mu)$,
which generates the $\nu$-segment $\tail(\mu)$. Now, it is a simple
matter to compute the finite $\nu$-segments starting at these
initial partitions and tabulate the values of $\dinv$ and $\mind$
for the resulting objects.  We thereby obtain a two-line array,
which needs to be of the form
\[ F_{\mu}=\left[\begin{array}{rrrrrrrr}
\dinv: & d & d+1 & d+2 & \cdots & e-1 & e & \cdots \\
\mind: &w_d&w_{d+1} & w_{d+2} & \cdots & w_{e-1} & w_e & \cdots
\end{array}\right], \]
where $w_i=\mind(\gamma(i))$ and $i=\dinv(\gamma(i))$ for all $i\geq d$.
We can terminate the display at $e=\dinv(\TI(\mu))$, 
since the right end of the array (starting with the values for $\TI(\mu)$) 
is known to be an infinite $(e,0,w_e)$-staircase by 
Proposition~\ref{prop:nu-stair}.

We now show that the global chain $\C_{\mu}$ can be decomposed into
an overlapping union of local chains $\calS_i$ in at most one way.
This decomposition is readily deduced from the word $w=w_dw_{d+1}\cdots w_e$. 
On one hand, we know the local chain decomposition must begin with the
exceptional local chain $\calS_0=\{\gamma(d)\}$ and end with
the exceptional local chain $\calS_c=\tail(\mu)=\nu^*(\TI(\mu))$.
On the other hand, we can uniquely build the local chains $\calS_i$ for 
$i=0,1,2,\ldots,c$ as follows. (Keep in mind Figure~\ref{fig:lc-graph},
especially the arrows showing ascents and descents forced by the
definition of local chains.) 
Scan $w$ from left to right, looking for \emph{descent
positions} $i$ where $w_i>w_{i+1}$. Each such descent marks a place
where the middle part of the current local chain ends and the
left part of the next local chain begins. The length of this new left
part is the unique $m$ such that $w_{i+1}=w_{i+2}=\cdots=w_{i+m}<w_{i+m+1}$.
Also, the right part of the old local chain equals the left part
of the new local chain. This process determines the values
of $a$, $a'$, $m$, $m'$, $h$, and $h'$ for each local chain $\calS_i$.
We must also check that the restriction of $F_{\mu}$ to each subinterval
$\{a,a+1,\ldots,a'-1\}$ is an $(a,m,h)$-staircase, as required by
the definition of local chains.

\begin{example}
Let $k=6$ and $\mu=(42)$, so $\TI(\mu)=(554221)$. Given the global chain
\[ \C_{(42)}=\nu^*(3111111)\cup \nu^*(42221)\cup \nu^*(44411)
            \cup \nu^*(554221), \]
let us find the constituent local chains for $\C_{\mu}$. The array
of $(\dinv,\mind)$ values for the beginning of this chain is:
\begin{equation}\label{eq:F42-array}
 F_{(42)}=\left[\begin{array}{r|rr|rrr|rrrrr|rrr}
\dinv: & 3 & 4 & 5 & 6 & 7 & 8 & 9 & 10 & 11 & 12 & 13 & 14 & \cdots \\
\mind: & 8 & 9 & 6 & 7 & 7 & 7 & 7 & 7  & 7  & 8  & 7 &  8  & \cdots
\end{array}\right], 
\end{equation}
where the bars show where $\nu$-segments begin and end.  
The word of $\mind$-values (with descents marked) is 
$w=8,9>6,\underline{7}^6,8>7,\underline{8}^7,\underline{9}^8,\cdots$.
Following the procedure above, we find the local chains
\begin{align*}
 \calS_0 &=\{(3111111)\},\\
 \calS_1 &=\nu^*(3111111)\cup\{(42221)\},\\
 \calS_2 &=\nu^*(42221)\cup\nu^*(44411)\cup\{(554221)\},\\
 \calS_3 &=\nu^*(554221)=\tail(\mu). 
\end{align*}
Note that $\lpart(\calS_2)\cup\mpart(\calS_2)$ is the union of
two $\nu$-segments, and the necessary staircase property does hold.
Because of the overlapping left and right parts of consecutive local
chains, we can conveniently present the values of $a,a',m,m',h,h'$
for all the local chains in a table such as the following:

\begin{center}
\begin{tabular}{ll|lllll}
 & & $\calS_0$ & $\calS_1$ & $\calS_2$ & $\calS_3$ & \\\hline
$a$ & $a'$ & --- & 3 & 5 & 13 & --- \\
$m$ & $m'$ & --- & 0 & 0 & 0  & --- \\
$h$ & $h'$ & --- & 8 & 6 & 7  & --- \\\hline
\end{tabular}
\end{center}

The entries in columns $1$ and $2$ mean that $a,m,h$ are undefined 
for $\calS_0$, while $a'=3$, $m'=0$, $h'=8$ for $\calS_0$. Reading
columns $2$ and $3$, we see that $(a,m,h)=(3,0,8)$ for $\calS_1$,
whereas $(a',m',h')=(5,0,6)$ for $\calS_1$. Next we find
$(a,m,h)=(5,0,6)$ and $(a',m',h')=(13,0,7)$ for $\calS_2$.
Finally, $(a,m,h)=(13,0,7)$ while $(a',m',h')$ are undefined
for $\calS_3$. All of these values were found from inspection
of~\eqref{eq:F42-array} (compare to Figure~\ref{fig:lc-graph}).  
For brevity, we use a shorter version of the table in the appendix.
In this example, the abbreviated version consists of the three 
vectors $a=(3,5,13)$, $m=(0,0,0)$, and $h=(8,6,7)$.
\end{example}

\subsection{Verifying the Local Opposite Property}
\label{subsec:verify-local-opp}

Our next example shows how to check the local opposite property for
the local chains comprising given global chains $\C_{\mu}$ and $\C_{\mu^*}$.

\begin{example}
For $\mu=(42)$, we have $\mu^*=(411)$, $\TI(411)=(6654211)$, and
\[ \C_{(411)}=\nu^*(221111)\cup\nu^*(33211)\cup\nu^*(6654211). \]
Proceeding as we did above, we find
\begin{equation}\label{eq:F411-array}
 F_{(411)}=\left[\begin{array}{r|rr|rrrrrrrr|rrr}
\dinv:& 2 & 3 & 4 & 5 & \cdots & 10 & 11 & \cdots & 17 & 18 & 19 & 20 & \cdots\\
\mind:& 7 & 8 & 6 & 7 & \cdots & 7  & 8  & \cdots & 8  &  9 &  8 &  9 & \cdots
\end{array}\right], 
\end{equation}
leading to local chains $\calT_i$ with parameters shown here:

\begin{center}
\begin{tabular}{ll|lllll}
 & & $\calT_0$ & $\calT_1$ & $\calT_2$ & $\calT_3$ & \\\hline
$a$ & $a'$ & --- & 2 & 4 & 19 & --- \\
$m$ & $m'$ & --- & 0 & 0 & 0  & --- \\
$h$ & $h'$ & --- & 7 & 6 & 8  & --- \\\hline
\end{tabular}
\end{center}

To see that $\C_{(411)}$ is globally opposite to $\C_{(42)}$,
we check that $\calT_{3-i}$ is locally opposite to $\calS_i$
for $i=0,1,2,3$, as follows. First, note that the $m$-vector
for $(411)$ is the reverse of the $m$-vector for $(42)$, and
the $h$-vector for $(411)$ is the reverse of the $h$-vector for $(42)$.
Second, note that the first object in $\C_{(42)}$ has
dinv $3=\ell(411)$, while the first object in $\C_{(411)}$ has
dinv $2=\ell(42)$. Third, we directly verify equations~\eqref{eq:loc-opp}
by computing
\[ 3+0+6+19=28=\binom{8}{2};\ \ 
   5+0+6+4 =15=\binom{6}{2};\ \ 
   13+0+6+2=21=\binom{7}{2}. \] 
\end{example}

\begin{example}
For $\mu$ of size $4$, we present the global chains $\C_{\mu}$ 
from the appendix to~\cite{LLL18} and confirm the local opposite property.
We list the initial objects of each $\nu$-segment of $\C_{\mu}$,
followed by the local chain parameters in abbreviated form.
We can check by inspection the reversal properties of the $m$-vectors 
and $h$-vectors, as well as the equality of the least
dinv value in $\C_{\mu}$ and the length of $\mu^*$.
We also show the verification of~\eqref{eq:loc-opp} in each case.  

\smallskip\noindent\emph{Involution on partitions of $4$:}
 $(4)^*=(4)$; $(31)^*=(22)$; $(211)^*=(1111)$.

\smallskip\noindent\emph{Global Chain $\C_{(4)}$:}
$(11111)$, $(2221)$, $(3331)$, $(44321)$.  
\\{}Local chain parameters:
$a=(1,3,10)$, $m=(0,0,0)$, $h=(6,5,6)$.

This chain is self-opposite, so the $m$-vector and $h$-vector
are palindromes.  We verify $1+0+4+10=\binom{6}{2}$, $3+0+4+3=\binom{5}{2}$,
and $10+0+4+1=\binom{6}{2}$ (this last check is redundant).

\medskip\noindent\emph{Global Chain $\C_{(31)}$:}
$(2211)$, $(44311)$.  
\\{}Local chain parameters:
$a=(2,9)$, $m=(0,0)$, $h=(5,6)$.

\noindent\emph{Global Chain $\C_{(22)}$:}
$(21111)$, $(3221)$.  
\\{}Local chain parameters:
$a=(2,4)$, $m=(0,0)$, $h=(6,5)$.

We verify $2+0+4+4=\binom{5}{2}$ and $9+0+4+2=\binom{6}{2}$.

\medskip\noindent\emph{Global Chain $\C_{(211)}$:}
$(32111)$, $(43111)$, $(44211)$.  
\\{}Local chain parameters:
$a=(4,6,8)$, $m=(0,0,0)$, $h=(6,6,6)$.

\noindent\emph{Global Chain $\C_{(1111)}$:}
$(31111)$, $(42111)$, $(43211)$.  
\\{}Local chain parameters:
$a=(3,5,7)$, $m=(0,0,0)$, $h=(6,6,6)$.

We verify $4+0+4+7=6+0+4+5=8+0+4+3=\binom{6}{2}$.
\end{example}

\subsection{Appendix: Chain Data}
\label{subsec:appendix}

This appendix lists the global chains and values of $a$, $m$, $h$
for all deficit partitions $\mu$ with $7\leq |\mu|\leq 9$. The
online extended appendix~\cite{online-app}
presents this information for $k=10$ and $k=11$.
In the data below, initial objects that do \emph{not} start new local
chains are marked $N$.

{\footnotesize
\smallskip\noindent\emph{Involution on partitions of $7$:}
 $(7)^*=(7)$; $(31111)^*=(3211)$; $(211111)^*=(1111111)$;
 $(61)^*=(331)$; $(52)^*=(52)$; $(511)^*=(4111)$; $(43)^*=(322)$;
 $(421)^*=(421)$; $(2221)^*=(2221)$; $(22111)^*=(22111)$.

\noindent 
\\$\C_{(7)}$: $(11111111)$, $(22222)$, $(33331)$, $(44432)$, $(555421)^N$, $(6664321)^N$, $(77654321)$.
\\ $a=(1,3,6,10,28)$, $m=(0,1,2,1,0)$, $h=(9,7,7,7,9)$.
\\$\C_{(31111)}$: $(3311111)$, $(5311111)$, $(5331111)$, $(6431111)$, $(6442111)$, $(6542211)$, $(77643211)$.
\\ $a=(4,6,8,10,12,14,24)$, $m=(0,0,0,0,0,0,0)$, $h=(8,8,8,8,8,8,9)$.
\\$\C_{(3211)}$: $(42111111)$, $(4421111)$, $(5422111)$, $(5532111)$, $(6533111)$, $(6643111)$, $(6644211)$.
\\ $a=(5,7,9,11,13,15,17)$, $m=(0,0,0,0,0,0,0)$, $h=(9,8,8,8,8,8,8)$.
\\$\C_{(211111)}$: $(43211111)$, $(54211111)$, $(54321111)$, $(65321111)$, $(65431111)$, $(75432111)$, $(76532111)$, $(76543111)$, $(77543211)$.
\\ $a=(7,9,11,13,15,17,19,21,23)$, $m=(0,0,0,0,0,0,0,0,0)$, $h=(9,9,9,9,9,9,9,9,9)$.
\\$\C_{(1111111)}$: $(43111111)$, $(53211111)$, $(54311111)$, $(64321111)$, $(65421111)$, $(65432111)$, $(76432111)$, $(76542111)$, $(76543211)$.
\\ $a=(6,8,10,12,14,16,18,20,22)$, $m=(0,0,0,0,0,0,0,0,0)$, $h=(9,9,9,9,9,9,9,9,9)$.
\\$\C_{(61)}$: $(511111)$, $(3333)$, $(44422)$, $(554421)^N$, $(77654311)$.
\\ $a=(3,5,9,27)$, $m=(0,2,1,0)$, $h=(7,7,7,9)$.
\\$\C_{(331)}$: $(21111111)$, $(32222)$, $(43331)$, $(544311)$.
\\ $a=(2,4,7,11)$, $m=(0,1,2,0)$, $h=(9,7,7,7)$.
\\$\C_{(52)}$: $(2211111)$, $(33221)$, $(555311)^N$, $(6654221)$.
\\ $a=(2,4,19)$, $m=(0,0,0)$, $h=(8,6,8)$.
\\$\C_{(511)}$: $(32111111)$, $(541111)$, $(44322)$, $(77654211)$.
\\ $a=(4,6,8,26)$, $m=(0,0,1,0)$, $h=(9,7,7,9)$.
\\$\C_{(4111)}$: $(31111111)$, $(42222)$, $(533211)$, $(77653211)$.
\\ $a=(3,5,8,25)$, $m=(0,1,0,0)$, $h=(9,7,7,9)$.
\\$\C_{(43)}$: $(322111)$, $(441111)$, $(44222)$, $(554111)$, $(553321)$.
\\ $a=(3,5,7,10,12)$, $m=(0,0,1,0,0)$, $h=(7,7,7,7,7)$.
\\$\C_{(322)}$: $(222111)$, $(332111)$, $(43222)$, $(544111)$, $(553221)$.
\\ $a=(2,4,6,9,11)$, $m=(0,0,1,0,0)$, $h=(7,7,7,7,7)$.
\\$\C_{(421)}$: $(4111111)$, $(522111)$, $(443111)$, $(552211)$, $(6653311)$.
\\ $a=(3,5,7,9,18)$, $m=(0,0,0,0,0)$, $h=(8,7,7,7,8)$.
\\$\C_{(2221)}$: $(521111)$, $(433111)$, $(552111)$, $(543311)$.
\\ $a=(4,6,8,10)$, $m=(0,0,0,0)$, $h=(7,7,7,7)$.
\\$\C_{(22111)}$: $(4221111)$, $(532211)$, $(6553211)$.
\\ $a=(5,7,16)$, $m=(0,0,0)$, $h=(8,7,8)$.

\smallskip\noindent\emph{Involution on partitions of $8$:}
 $(8)^*=(8)$; $(4211)^*=(4211)$; $(41111)^*=(41111)$;
 $(32111)^*=(32111)$; $(311111)^*=(311111)$; $(2111111)^*=(11111111)$;
$(71)^*=(44)$; $(62)^*=(5111)$; $(611)^*=(521)$; $(53)^*=(2222)$;
$(3221)^*=(422)$; $(431)^*=(332)$; $(3311)^*=(3311)$; 
$(22211)^*=(22211)$; $(221111)^*=(221111)$.

\noindent 
\\$\C_{(8)}$: $(111111111)$, $(222221)$, $(33332)$, $(444321)$, $(555431)^N$, $(6665321)^N$, $(77754321)^N$, $(887654321)$.
\\ $a=(1,3,6,10,36)$, $m=(0,0,1,0,0)$, $h=(10,7,7,7,10)$.
\\$\C_{(4211)}$: $(33111111)$, $(5411111)$, $(5332111)$, $(6531111)$, $(6443111)$, $(6642211)$, $(77644211)$.
\\ $a=(4,6,8,10,12,14,24)$, $m=(0,0,0,0,0,0,0)$, $h=(9,8,8,8,8,8,9)$.
\\$\C_{(41111)}$: $(421111111)$, $(4431111)$, $(6422111)$, $(5542111)$, $(6533211)$, $(887643211)$.
\\ $a=(5,7,9,11,13,32)$, $m=(0,0,0,0,0,0)$, $h=(10,8,8,8,8,10)$.
\\$\C_{(32111)}$: $(42211111)$, $(44211111)$, $(54221111)$, $(55321111)$, $(65331111)$, $(75431111)$, $(75532111)$, $(76533111)$, $(77543111)$, $(77553211)$.
\\ $a=(5,7,9,11,13,15,17,19,21,23)$, $m=(0,0,0,0,0,0,0,0,0,0)$, $h=(9,9,9,9,9,9,9,9,9,9)$.
\\$\C_{(311111)}$: $(431111111)$, $(53311111)$, $(64311111)$, $(64421111)$, $(65422111)$, $(66432111)$, $(76442111)$, $(76542211)$, $(887543211)$.
\\ $a=(6,8,10,12,14,16,18,20,31)$, $m=(0,0,0,0,0,0,0,0,0)$, $h=(10,9,9,9,9,9,9,9,10)$.
\\$\C_{(2111111)}$: $(532111111)$, $(543111111)$, $(643211111)$, $(654211111)$, $(654321111)$, $(764321111)$, $(765421111)$, $(765432111)$, $(875432111)$, $(876532111)$, $(876543111)$, $(886543211)$.
\\ $a=(8,10,12,14,16,18,20,22,24,26,28,30)$, $m=(\underline{0}^{12})$, 
$h=(\underline{10}^{12})$.
\\$\C_{(11111111)}$: $(432111111)$, $(542111111)$, $(543211111)$, $(653211111)$, $(654311111)$, $(754321111)$, $(765321111)$, $(765431111)$, $(865432111)$, $(876432111)$, $(876542111)$, $(876543211)$.
\\ $a=(7,9,11,13,15,17,19,21,23,25,27,29)$, 
$m=(\underline{0}^{12})$, $h=(\underline{10}^{12})$.
\\$\C_{(71)}$: $(222211)$, $(33322)$, $(444311)$, $(555331)^N$, $(6655321)^N$, $(887654311)$.
\\ $a=(2,5,9,35)$, $m=(0,1,0,0)$, $h=(7,7,7,10)$.
\\$\C_{(44)}$: $(211111111)$, $(322221)$, $(43332)$, $(544321)$.
\\ $a=(2,4,7,11)$, $m=(0,0,1,0)$, $h=(10,7,7,7)$.
\\$\C_{(62)}$: $(321111111)$, $(551111)$, $(44332)$, $(6664311)^N$, $(77654221)$.
\\ $a=(4,6,8,26)$, $m=(0,0,1,0)$, $h=(10,7,7,9)$.
\\$\C_{(5111)}$: $(22111111)$, $(33222)$, $(443211)$, $(887653211)$.
\\ $a=(2,4,7,33)$, $m=(0,1,0,0)$, $h=(9,7,7,10)$.
\\$\C_{(611)}$: $(41111111)$, $(522211)$, $(444211)$, $(554331)^N$, $(887654211)$.
\\ $a=(3,5,8,34)$, $m=(0,0,0,0)$, $h=(9,7,7,10)$.
\\$\C_{(521)}$: $(311111111)$, $(422221)$, $(533311)$, $(554411)^N$, $(77653311)$.
\\ $a=(3,5,8,25)$, $m=(0,0,0,0)$, $h=(10,7,7,9)$.
\\$\C_{(53)}$: $(422211)$, $(444111)$, $(552221)$, $(555221)^N$, $(6653321)$.
\\ $a=(4,7,9,18)$, $m=(0,0,0,0)$, $h=(7,7,7,8)$.
\\$\C_{(2222)}$: $(2221111)$, $(333111)$, $(432221)$, $(543221)$.
\\ $a=(2,4,6,9)$, $m=(0,0,0,0)$, $h=(8,7,7,7)$.
\\$\C_{(3221)}$: $(3221111)$, $(432211)$, $(6553111)$, $(6643311)$.
\\ $a=(3,5,14,16)$, $m=(0,0,0,0)$, $h=(8,7,8,8)$.
\\$\C_{(422)}$: $(5211111)$, $(4331111)$, $(542221)$, $(555211)^N$, $(6653221)$.
\\ $a=(4,6,8,17)$, $m=(0,0,0,0)$, $h=(8,8,7,8)$.
\\$\C_{(431)}$: $(322211)$, $(442211)$, $(6653111)$, $(6644311)$.
\\ $a=(3,6,15,17)$, $m=(0,0,0,0)$, $h=(7,7,8,8)$.
\\$\C_{(332)}$: $(5111111)$, $(5221111)$, $(532221)$, $(544221)$.
\\ $a=(3,5,7,10)$, $m=(0,0,0,0)$, $h=(8,8,7,7)$.
\\$\C_{(3311)}$: $(3321111)$, $(43322)$, $(6554211)$.
\\ $a=(4,6,16)$, $m=(0,1,0)$, $h=(8,7,8)$.
\\$\C_{(22211)}$: $(4411111)$, $(5322111)$, $(5531111)$, $(6433111)$, $(6642111)$, $(6544211)$.
\\ $a=(5,7,9,11,13,15)$, $m=(0,0,0,0,0,0)$, $h=(8,8,8,8,8,8)$.
\\$\C_{(221111)}$: $(53111111)$, $(6421111)$, $(5442111)$, $(6532211)$, $(76643211)$.
\\ $a=(6,8,10,12,22)$, $m=(0,0,0,0,0)$, $h=(9,8,8,8,9)$.

\smallskip\noindent\emph{Involution on partitions of $9$:}
$(9)^*=(9)$, $(54)^*=(54)$, $(531)^*=(432)$, $(522)^*=(33111)$,
$(5211)^*=(51111)$, $(4311)^*=(4221)$, $(42111)^*=(411111)$,
$(3321)^*=(3321)$, $(32211)^*=(3222)$, $(321111)^*=(321111)$,
$(3111111)^*=(3111111)$, $(22221)^*=(22221)$, $(2211111)^*=(222111)$,
$(21111111)^*=(21111111)$, $(111111111)^*=(111111111)$,
$(81)^*=(441)$, $(72)^*=(621)$, $(711)^*=(711)$, $(63)^*=(63)$,
$(6111)^*=(333)$.

\noindent
$\C_{(9)}$: $(1111111111)$, $(222222)$, $(333321)$, $(444421)^N$, $(555432)$, $(6665421)^N$, $(77764321)^N$, $(888654321)^N$, $(9987654321)$.
\\ $a=(1,3,6,15,45)$, $m=(0,1,0,1,0)$, $h=(11,8,7,8,11)$.
\\$\C_{(54)}$: $(2222111)$, $(3331111)$, $(442221)$, $(555111)^N$, $(553331)^N$, $(6654111)$, $(6644321)$.
\\ $a=(2,4,6,15,17)$, $m=(0,0,0,0,0)$, $h=(8,8,7,8,8)$.
\\$\C_{(531)}$: $(3222111)$, $(6311111)$, $(6331111)$, $(6441111)$, $(554222)$, $(6652211)$, $(77644311)$.
\\ $a=(3,5,7,9,11,14,24)$, $m=(0,0,0,0,1,0,0)$, $h=(8,8,8,8,8,8,9)$.
\\$\C_{(432)}$: $(51111111)$, $(5222111)$, $(532222)$, $(5533111)$, $(6633111)$, $(6644111)$, $(6644221)$.
\\ $a=(3,5,7,10,12,14,16)$, $m=(0,0,1,0,0,0,0)$, $h=(9,8,8,8,8,8,8)$.
\\$\C_{(522)}$: $(52211111)$, $(533221)$, $(6664211)^N$, $(77653221)$.
\\ $a=(5,7,24)$, $m=(0,0,0)$, $h=(9,7,9)$.
\\$\C_{(33111)}$: $(32211111)$, $(433211)$, $(76653211)$.
\\ $a=(3,5,22)$, $m=(0,0,0)$, $h=(9,7,9)$.
\\$\C_{(5211)}$: $(4211111111)$, $(4432111)$, $(6522111)$, $(5543111)$, $(6633211)$, $(887644211)$.
\\ $a=(5,7,9,11,13,32)$, $m=(0,0,0,0,0,0)$, $h=(11,8,8,8,8,10)$.
\\$\C_{(51111)}$: $(331111111)$, $(6411111)$, $(6332111)$, $(6541111)$, $(6443211)$, $(9987643211)$.
\\ $a=(4,6,8,10,12,41)$, $m=(0,0,0,0,0,0)$, $h=(10,8,8,8,8,11)$.
\\$\C_{(4311)}$: $(52111111)$, $(4422111)$, $(542222)$, $(6442211)$, $(77643111)$, $(77554211)$.
\\ $a=(4,6,8,11,21,23)$, $m=(0,0,1,0,0,0)$, $h=(9,8,8,8,9,9)$.
\\$\C_{(4221)}$: $(33211111)$, $(54111111)$, $(5522111)$, $(553222)$, $(6552211)$, $(77643311)$.
\\ $a=(4,6,8,10,13,23)$, $m=(0,0,0,1,0,0)$, $h=(9,9,8,8,8,9)$.
\\$\C_{(42111)}$: $(4311111111)$, $(53321111)$, $(65311111)$, $(64431111)$, $(75422111)$, $(66532111)$, $(76443111)$, $(77542211)$, $(887553211)$.
\\ $a=(6,8,10,12,14,16,18,20,31)$, $m=(0,0,0,0,0,0,0,0,0)$, $h=(11,9,9,9,9,9,9,9,10)$.
\\$\C_{(411111)}$: $(422111111)$, $(44311111)$, $(64221111)$, $(55421111)$, $(65332111)$, $(76431111)$, $(75542111)$, $(76533211)$, $(9987543211)$.
\\ $a=(5,7,9,11,13,15,17,19,40)$, $m=(0,0,0,0,0,0,0,0,0)$, $h=(10,9,9,9,9,9,9,9,11)$.
\\$\C_{(3321)}$: $(6211111)$, $(433311)$, $(544411)^N$, $(6553311)$.
\\ $a=(4,6,15)$, $m=(0,0,0)$, $h=(8,7,8)$.
\\$\C_{(32211)}$: $(4222111)$, $(6321111)$, $(5441111)$, $(6432211)$, $(76643111)$, $(77544211)$.
\\ $a=(4,6,8,10,20,22)$, $m=(0,0,0,0,0,0)$, $h=(8,8,8,8,9,9)$.
\\$\C_{(3222)}$: $(44111111)$, $(53221111)$, $(5433111)$, $(6632111)$, $(6544111)$, $(6643221)$.
\\ $a=(5,7,9,11,13,15)$, $m=(0,0,0,0,0,0)$, $h=(9,9,8,8,8,8)$.
\\$\C_{(321111)}$: $(531111111)$, $(533111111)$, $(643111111)$, $(644211111)$, $(654221111)$, $(664321111)$, $(764421111)$, $(765422111)$, $(775432111)$, $(875532111)$, $(876533111)$, $(886543111)$, $(886643211)$.
\\ $a=(6,8,10,12,14,16,18,20,22,24,26,28,30)$, 
$m=(\underline{0}^{13})$,
$h=(\underline{10}^{13})$.
\\$\C_{(3111111)}$: $(4321111111)$, $(542211111)$, $(553211111)$, $(653311111)$, $(754311111)$, $(755321111)$, $(765331111)$, $(865431111)$, $(866432111)$, $(876442111)$, $(876542211)$, $(9986543211)$.
\\ $a=(7,9,11,13,15,17,19,21,23,25,27,39)$, 
$m=(\underline{0}^{12})$, 
$h=(11,\underline{10}^{10},11)$.
\\$\C_{(22221)}$: $(4322111)$, $(5521111)$, $(543222)$, $(6552111)$, $(6543311)$.
\\ $a=(5,7,9,12,14)$, $m=(0,0,1,0,0)$, $h=(8,8,8,8,8)$.
\\$\C_{(2211111)}$: $(43311111)$, $(64211111)$, $(54421111)$, $(65322111)$, $(66431111)$, $(75442111)$, $(76532211)$, $(877543211)$.
\\ $a=(6,8,10,12,14,16,18,29)$, $m=(0,0,0,0,0,0,0,0)$, $h=(9,9,9,9,9,9,9,10)$.
\\$\C_{(222111)}$: $(442111111)$, $(55311111)$, $(64331111)$, $(75421111)$, $(65532111)$, $(76433111)$, $(77542111)$, $(76553211)$.
\\ $a=(7,9,11,13,15,17,19,21)$, $m=(0,0,0,0,0,0,0,0)$, $h=(10,9,9,9,9,9,9,9)$.
\\$\C_{(2\underline{1}^7)}$: $(532\underline{1}^7)$, $(543\underline{1}^7)$, 
$(6432\underline{1}^6)$, $(6542\underline{1}^6)$, $(65432\underline{1}^5)$, 
$(76432\underline{1}^5)$, $(7654211111)$, $(7654321111)$, $(8754321111)$, $(8765321111)$, $(8765431111)$, $(9765432111)$, $(9875432111)$, $(9876532111)$, $(9876543111)$, $(9976543211)$.
\\ $a=(8,10,12,14,16,18,20,22,24,26,28,30,32,34,36,38)$, 
$m=(\underline{0}^{16})$,
$h=(\underline{11}^{16})$.
\\$\C_{(\underline{1}^9)}$: $(542\underline{1}^7)$, $(5432\underline{1}^6)$, 
$(6532\underline{1}^6)$, $(6543\underline{1}^6)$, $(75432\underline{1}^5)$, 
$(76532\underline{1}^5)$, $(7654311111)$, $(8654321111)$, $(8764321111)$, $(8765421111)$, $(8765432111)$, $(9865432111)$, $(9876432111)$, $(9876542111)$, $(9876543211)$.
\\ $a=(9,11,13,15,17,19,21,23,25,27,29,31,33,35,37)$, 
$m=(\underline{0}^{15})$,
$h=(\underline{11}^{15})$.
\\$\C_{(81)}$: $(6111111)$, $(333311)$, $(44442)^N$, $(555422)$, $(6664421)^N$, $(77664321)^N$, $(9987654311)$.
\\ $a=(3,5,14,44)$, $m=(0,0,1,0)$, $h=(8,7,8,11)$.
\\$\C_{(441)}$: $(2111111111)$, $(322222)$, $(433321)$, $(544421)^N$, $(6554311)$.
\\ $a=(2,4,7,16)$, $m=(0,1,0,0)$, $h=(11,8,7,8)$.
\\$\C_{(72)}$: $(411111111)$, $(522221)$, $(444221)$, $(554431)^N$, $(77754311)^N$, $(887654221)$.
\\ $a=(3,5,8,34)$, $m=(0,0,0,0)$, $h=(10,7,7,10)$.
\\$\C_{(621)}$: $(221111111)$, $(332221)$, $(443311)$, $(6655311)^N$, $(887653311)$.
\\ $a=(2,4,7,33)$, $m=(0,0,0,0)$, $h=(10,7,7,10)$.
\\$\C_{(711)}$: $(3111111111)$, $(422222)$, $(44441)$, $(555322)$, $(6654421)^N$, $(9987654211)$.
\\ $a=(3,5,8,13,43)$, $m=(0,1,3,1,0)$, $h=(11,8,8,8,11)$.
\\$\C_{(63)}$: $(22211111)$, $(333211)$, $(533321)$, $(555411)^N$, $(6664221)^N$, $(77653321)$.
\\ $a=(2,4,8,25)$, $m=(0,0,0,0)$, $h=(9,7,7,9)$.
\\$\C_{(6111)}$: $(332211)$, $(554322)$, $(9987653211)$.
\\ $a=(3,12,42)$, $m=(0,1,0)$, $h=(7,8,11)$.
\\$\C_{(333)}$: $(3211111111)$, $(432222)$, $(543321)$.
\\ $a=(4,6,9)$, $m=(0,1,0)$, $h=(11,8,7)$.
}


\end{document}